\documentclass[a4paper,reqno,11pt]{amsart}
\usepackage{amssymb,amsmath,amsfonts,amsthm,stmaryrd}
\usepackage{mathrsfs}

\numberwithin{equation}{section}
\usepackage[left=1 in, right=1 in,top=1 in, bottom=1 in]{geometry}

\providecommand{\abs}[1]{\left\vert#1\right\vert}
\providecommand{\norm}[1]{\left\Vert#1\right\Vert}
\providecommand{\pnorm}[2]{\left\Vert#1\right\Vert_{L^{#2}}}

\providecommand{\Rn}[1]{\mathbb{R}^{#1}}

\providecommand{\sd}[1]{\mathcal{D}_{#1}}
\providecommand{\se}[1]{\mathcal{E}_{#1}}
\providecommand{\sdb}[1]{\bar{\mathfrak{D}}_{#1}}
\providecommand{\seb}[1]{\bar{\mathfrak{E}}_{#1}}
\providecommand{\sdm}[1]{\mathbb{D}_{#1}}
\providecommand{\sem}[1]{\mathbb{E}_{#1}}

\providecommand{\fd}[1]{\mathfrak{D}_{#1}}
\providecommand{\fe}[1]{\mathfrak{E}_{#1}}
\providecommand{\fdb}[1]{\bar{\mathfrak{D}}_{#1}}
\providecommand{\feb}[1]{\bar{\mathfrak{E}}_{#1}}

\providecommand{\ns}[1]{\norm{#1}^2}
\providecommand{\as}[1]{\abs{#1}^2}
\providecommand{\pns}[2]{\norm{#1}^2_{L^{#2}}}

\def\senn{\fe{N+2,2}}
\def\sdnn{\fd{N+2,2}}

\def\nab{\nabla}
\def\al{\alpha}
\def\dt{\partial_t}

\def\hal{\frac{1}{2}}
\def\ep{\varepsilon}

\def\ls{\lesssim}
\def\p{\partial}
\def\pa{\partial^\alpha}
\def\sg{\mathbb{D}}

\def\da{\Delta_{\mathcal{A}}}
\def\naba{\nab_{\mathcal{A}}}
\def\diva{\diverge_{\mathcal{A}}}
\def\Sa{S_{\mathcal{A}}}

\def\H1{{_0}H^1(\Omega)}

\def\cdott{\!\cdot\!}

\def\a{\mathcal{A}}
\def\m{\mathcal{M}}
\def\f{\mathcal{F}_{2N}}
\def\g{\mathcal{G}_{2N}}

\def\i{\mathcal{I}}

\def\fj1{\mathcal{J}^{-1}}
\def\k{\mathfrak{K}}
\def\n{\mathcal{N}}

\newcommand{\beq}{\begin{equation}}
\newcommand{\eeq}{\end{equation}}
\renewcommand{\(}{\left(}
 \renewcommand{\)}{\right)}
\renewcommand{\[}{\left[}
 \renewcommand{\]}{\right]}

\newcommand{\D}{\partial^{\a}}

\newcommand{\tQ}{\widetilde Q}

\def\rest{\hskip 1pt{\hbox to 10.8pt{\hfill\vrule height 7pt width 0.4pt depth 0pt\hbox{\vrule height 0.4pt
width 7.6pt depth 0pt}\hfill}}}

\def\evalu{\hskip 1pt{\hbox to 2pt{\hfill \vrule height -6pt width 0.4pt depth0pt}}}

\DeclareMathOperator{\diverge}{div}

\newtheorem{lem}{Lemma}[section]

\newtheorem{prop}[lem]{Proposition}
\newtheorem{thm}[lem]{Theorem}
\newtheorem{remark}[lem]{Remark}

\title[Viscous surface waves without surface tension]{Anisotropic decay and global well-posedness of viscous surface waves without surface tension}

\author{Yanjin Wang}
\address{
School of Mathematical Sciences\\
Xiamen University\\
Xiamen, Fujian 361005, China}
\email[Y. J. Wang]{yanjin$\_$wang@xmu.edu.cn}

\thanks{This research was supported by the National Natural Science Foundation of China (No. 11771360, 11531010) and the Natural Science Foundation of Fujian Province of China (No. 2019J02003).}
\date{\today}


\begin{document}
\begin{abstract}
We consider a viscous incompressible fluid below the air and above a fixed bottom. The fluid dynamics is governed by the gravity-driven incompressible Navier-Stokes equations, and the effect of surface tension is neglected on the free surface. The global well-posedness and long-time behavior of solutions near equilibrium have been intriguing questions since Beale (\emph{Comm. Pure Appl. Math.} 34  (1981), no. 3, 359--392). It had been thought that certain low frequency assumption of the initial data is needed to derive  an integrable decay rate of the velocity so that the global solutions in $3D$ can be constructed, while the global well-posedness in $2D$ was left open. In this paper, by exploiting the anisotropic decay rates of the velocity, which are even not integrable, we prove the global well-posedness in both $2D$ and $3D$, without any low frequency assumption of the initial data. One of key observations here is a  cancelation in nonlinear estimates of the viscous stress tensor term in the bulk by using Alinhac good unknowns, when estimating the energy evolution of the highest order horizontal spatial derivatives of the solution.
\end{abstract}

\maketitle
\setcounter{tocdepth}{2}



\section{Introduction}

\subsection{Formulation in Eulerian coordinates}
We consider a viscous, incompressible fluid evolving in a $d$-dimensional moving domain
\begin{equation}
\Omega(t) = \{ y \in   \Rn{d} \;\vert\; -b< y_d  < \eta(y_h ,t)\}.
\end{equation}
Here the dimension $d=2,3$, and $y=(y_h,y_d)$ for $y_h=(y_1,y_{d-1})\in \Rn{d-1}$ the horizontal coordinate and $y_d$ the vertical one. The lower boundary of $\Omega(t)$, denoted by $\Sigma_b$,  is assumed to be rigid and given, but the upper boundary, denoted by $\Sigma(t)$, is a free surface that is the graph of the unknown function $\eta: \Rn{d-1}\times\Rn{+} \to \Rn{}$. We assume that $b>0$ is a fixed constant, so that the lower boundary is flat. The fluid is described by its velocity and pressure functions, which are given for each $t\ge0$ by $ u(\cdot,t):\Omega(t) \to \Rn{d} $ and $ p(\cdot,t):\Omega(t) \to \Rn{}$, respectively.  For each $t>0$ we require that $(u, p, \eta)$ satisfy the gravity-driven free-surface incompressible Navier-Stokes equations:
\begin{equation}\label{ns_euler}
\begin{cases}
\partial_t u + u \cdot \nabla u + \nabla p-\mu  \Delta u  = 0& \text{in }
\Omega(t) \\
\diverge{u}=0 & \text{in }\Omega(t) \\
(p I_d - \mu \sg u )\nu = g \eta \nu & \text{on } \Sigma(t) \\
\partial_t \eta = u_d - u_h\!\cdot\! D\eta &
\text{on }\Sigma(t) \\
u = 0 & \text{on }\Sigma_b.
\end{cases}%
\end{equation}
Here $\mu>0$ is the viscosity and $g>0$ is the strength of gravity. The tensor $p I_d - \mu \sg u $ is known as the viscous stress tensor for $I_d$ the $d \times d$ identity matrix and $ \mathbb{D} u = \nabla u + (\nabla u)^t$ the symmetric gradient of $u$, and $\nu=(-D\eta,1)/\sqrt{1+ \abs{D\eta}^2}$ is the outward-pointing unit normal on $\Sigma(t)$ for $ D  \eta$ for the horizontal gradient of $\eta$.
The fourth equation in \eqref{ns_euler} is called the kinematic boundary condition which implies that the free surface is advected with the fluid, where $u_h$ and $u_d$ are the horizontal and vertical components of the velocity, respectively.  Note that in \eqref{ns_euler} we have shifted the gravitational forcing from the bulk to the boundary and eliminated the constant atmospheric pressure, $p_{atm}$, in the usual way by adjusting the actual pressure $\bar{p}$ according to $p = \bar{p} + g y_d  - p_{atm}$. Without loss of generality, we may assume that $\mu = g = 1$.

To complete the formulation of the problem, we must specify the initial conditions. We suppose that the initial surface $\Sigma(0)$ is given by the graph of the function $\eta(0)=\eta_0: \Rn{d-1}\rightarrow \mathbb{R}$, which yields the initial domain $\Omega(0)$ on which we specify the initial data for the velocity, $u(0)=u_0: \Omega(0) \rightarrow \Rn{d}$. We will assume that $\eta_0 > -b$ on $\Rn{d-1}$ and that $(u_0, \eta_0)$ satisfy certain compatibility conditions, which we will describe later.

\subsection{Reformulation in flattening coordinates}
In order to work in a fixed domain, we want to flatten the free surface via a coordinate transformation. We will use a flattening transformation introduced by Beale in \cite{beale_2}.  To this end, we consider the fixed equilibrium domain
\begin{equation}
\Omega:= \{x \in \Rn{d} \; \vert\;  -b < x_d  < 0  \}
\end{equation}
for which we will write the coordinates as $x\in \Omega$.  We write $\Sigma:=\{x_d =0\}$ for the upper boundary of $\Omega$, and we view $\eta$ as a function on $\Sigma\times\mathbb{R}^+$.  We define \begin{equation}
 \bar{\eta}:= \mathcal{P} \eta = \text{harmonic extension of }\eta \text{ into the lower half space},
\end{equation}
where $\mathcal{P} \eta$ is defined by \eqref{poisson_def_inf}.  The harmonic extension $\bar{\eta}$ allows us to flatten the coordinate domain via the mapping
\begin{equation}\label{mapping_def}
 \Omega \ni x \mapsto   (x_h, \phi(x,t)) := \Phi(x,t) = (y_h ,y_d ) \in \Omega(t),
\end{equation}
where $\phi(x,t)=x_d+\varphi(x,t)$ for $\varphi(x,t)=  \tilde{b}  \bar{\eta}(x,t)$ with $\tilde{b}  = (1+x_d /b).$
Note that $\Phi(\Sigma,t) = \Sigma(t)$ and $\Phi(\cdot,t)\vert_{\Sigma_b} = Id_{\Sigma_b}$, i.e. $\Phi$ maps $\Sigma$ to the free surface and keeps the lower surface fixed.   We have
\begin{equation}\label{A_def}
 \nab \Phi =
\begin{pmatrix}
 I_{d-1}   & 0 \\
D\varphi & J
\end{pmatrix}
\text{ and }
 \mathcal{A} := (\nab \Phi^{-1})^T =
\begin{pmatrix}
 I_{d-1}   & -D\varphi K \\ 
 0   & K
\end{pmatrix}
\end{equation}
for
\begin{equation}\label{ABJ_def}
J =1+\p_d\varphi=  1+ \bar{\eta}/b + \p_d  \bar{\eta} \tilde{b} \text{ and } K = J^{-1}.
\end{equation}
Here $J = \det{\nab \Phi}$ is the Jacobian of the coordinate transformation.

If $\eta$ is sufficiently small (in an appropriate Sobolev space), then the mapping $\Phi$ is a diffeomorphism.  This allows us to transform the problem to one on the fixed spatial domain $\Omega$ for $t \ge 0$.  In the new coordinates, the system \eqref{ns_euler} becomes
\begin{equation}\label{geometric}
 \begin{cases}
  \dt^\a u   + u \cdot \naba u+ \naba p  -\da u     =0 & \text{in } \Omega \\
 \diva u = 0 & \text{in }\Omega \\
\Sa(p,u) \n = \eta \n  & \text{on } \Sigma \\
 \dt \eta = u \cdot \n & \text{on } \Sigma \\
 u = 0 & \text{on } \Sigma_b \\
 (u , \eta)\mid_{t=0} = (u_0,\eta_0).
 \end{cases}
\end{equation}
Here we have written the differential operators $\dt^\a, \naba$, $\diva$, and $\da$  by $ \dt^\a:=\dt   -  K\dt \varphi \p_d    $, $(\naba )_i := \a_{ij} \p_j $, $\diva  := \naba\cdot$, and $\da  := \diva \naba $.  We have also written  $\n := (-D\eta,  1)$ for the non-unit normal to $\Sigma(t)$ and
$\Sa(p,u)  = (p I_d  - \sg_{\a} u)$ for $ \sg_{\a} u : = \nabla_\a u+ (\nabla_\a u)^t$ the symmetric $\a-$gradient of $u$.  Note that if we extend $\diva$ to act on symmetric tensors in the natural way, then $\diva \Sa(p,u) = \naba p - \da u$ for  $\diva u=0$.
Recall that $\a$ is determined by $\eta$ through the relation \eqref{A_def}.  This means that all of the differential operators in \eqref{geometric} are connected to $\eta$, and hence to the geometry of the free surface.

\subsection{Previous results}
Free boundary problems in fluid mechanics have been studied by many authors in many different contexts. Here we will mention only the work most relevant to our present setting, that is, the viscous surface wave problem, which has attracted the attention of many mathematicians since the pioneering work of Beale \cite{beale_1}. 

In \cite{beale_1}, Beale  proved the local well-posedness of the viscous surface wave problem without surface tension, \eqref{ns_euler}, in Lagrangian coordinates: given $\Omega(0)=\Omega_0$ and $u_0\in H^{r-1}(\Omega_0)$ for $r\in(3,7/2)$, there exist  a time $T>0$ and a unique solution  on $[0,T]$ so that $v\in L^2(0,T; H^{r}( \Omega_0)) \cap H^{r/2}(0,T; L^2( \Omega_0))$, where $v=u\circ \zeta$ for $\zeta$ the Lagrangian flow map satisfying $\dt \zeta  =v$ in $\Omega_0$ and $\zeta(0)=Id$ in $\Omega_0$.  Beale \cite{beale_1}  also showed that for certain $\Theta \in H^r(\Omega_0)$ with $\Theta=0$ on $\Sigma_b$, there cannot exist a curve of solutions $ v^\ep $, defined for $\ep$ near $0$, with
$
 \zeta^\ep(0)  = Id + \ep \Theta$ and $v^\ep(0)  = 0 
$,
and $v^\ep $ is of the form $v^\ep  = \ep v^{(1)} + \ep^2 v^{(2)} + O(\ep^3) $, such that
\beq \label{vdecay}
v^\ep  \in L^1(0,\infty; H^r(\Omega_0))
\eeq
and
\begin{equation}\label{no_go_2}
 \lim_{t\to \infty} \zeta_3^\ep(t) \vert_{\Sigma} =0.
\end{equation}
This would suggest a nondecay theorem that a ``reasonable'' small-data global well-posedness with decay of the free surface is false and that any existence theorem for all time would necessarily have a more special hypothesis or a weaker conclusion than the assertion which was shown to be untrue in \cite{beale_1}. Thereafter, the global well-posedness and long-time behavior of solutions to \eqref{ns_euler} near equilibrium have been intriguing questions since \cite{beale_1}. Sylvester \cite{sylvester} and Tani and Tanaka \cite{tani_tanaka} studied the existence of small-data global-in-time solutions via the parabolic regularity method as \cite{beale_1}, and they make no claims about the decay of the solutions.  Sylvester \cite{sylvester2} discussed the decay of the solution for the linearized problem around equilibrium in $2D$.   As pointed out by Guo and Tice \cite{GT_per,GT_inf}, due to the growth in the highest order spatial derivatives of $\eta$ as will be seen later,  it seems impossible to construct global-in-time solutions to  \eqref{ns_euler}  without also deriving a decay result, at least by energy methods.

For the problem with surface tension, that is, the fourth equation in \eqref{ns_euler} is modified to be
\begin{equation}
 (p I_d - \mu \sg u ) \nu = g \eta \nu - \sigma H \nu\quad \text{ on }\Sigma(t),
\end{equation}
where $H=D\cdot\(D\eta/\sqrt{1+ \abs{D\eta}^2}\)$ is the mean curvature of the free surface  and $\sigma >0$ is the surface tension coefficient, it is conceivable that the regularizing effect of surface tension might lead to a global well-posedness valid for large time. Beale \cite{beale_2}  proved the global well-posedness for the problem with surface tension and with a curved bottom  in   flattening coordinates provided that  $u_0 \in H^{r-1/2}(\Omega)$ and $\eta_0 \in H^{r}(\Sigma)$ for $r\in(3,7/2)$ are sufficiently small.  Moreover, Beale and Nishida \cite{beale_nishida}  showed that for the flat bottom  in $3D$  if $\eta_0 \in L^1(\Sigma)$ is small, then the solution  constructed in \cite{beale_2} obeys
\begin{equation}\label{secc1}
 (1+t)^2 \ns{u(t)}_{H^2(\Omega)}  +   \sum_{j=0}^2 (1+t)^{j+1} \ns{D^j \eta (t)}_{L^2(\Sigma)}  < \infty 
\end{equation}
and that this decay rate is optimal. Note that if ignoring the different coordinates,  the decay of $\eta$ in \eqref{secc1} implies \eqref{no_go_2},  but the decay rate of $u$ is not sufficiently rapid to guarantee \eqref{vdecay}, even with surface tension.

If  the domain is horizontally periodic and assuming that $\eta_0$ has the zero average, then the situation is significantly different. Nishida, Teramoto, and Yoshihara  \cite{nishida_1} showed that for the problem with surface tension and with a flat bottom, there exists $\gamma>0$ so that
\begin{equation}
  e^{\gamma t} \left( \ns{u (t)}_{H^2(\Omega)} +   \ns{\eta (t)}_{H^3(\Sigma)} \right) < \infty.
\end{equation}
Hataya \cite{hataya} proved that for the problem without surface tension and with a flat bottom,  if   $u_0 \in H^{r-1}(\Omega)$ and $\eta_0 \in H^{r-1/2}(\Sigma)$ for $r\in (5,11/2)$ are sufficiently small, then there exists a unique global solution satisfying
\begin{equation}\label{secc2}
\int_{0}^{\infty }(1+t)^{2}\ns{u(t)}_{H^{r-1}(\Omega)} dt + (1+t)^{2} \ns{\eta(t)}_{H^{r-2}(\Sigma)} < \infty.
\end{equation}
Guo and Tice \cite{GT_per} showed that for  the problem without surface tension and with a curved bottom  if  $u_0 \in H^{4N}(\Omega)$ and $\eta_0 \in H^{4N+1/2}(\Sigma)$ for $N\ge 3$ are sufficiently small, then there exists a  unique global solution such that
\begin{equation}\label{secc3}
(1+t)^{4N-8}\(\ns{u(t)}_{H^{2N+4}(\Omega)}  +\ns{\eta (t)}_{H^{2N+4}(\Sigma)} \)< \infty.
\end{equation} 
Tan and Wang \cite{TW14} established the global-in-time zero surface tension limit of the problem with surface tension and with a curved bottom for the sufficiently small initial data.
We remark that the argument in Beale's nondecay theorem of \cite{beale_1}  works in horizontally periodic domains as well, 
and  Guo and Tice \cite{GT_per}  showed that the zero average of $\eta_0$ prevents the choice of $\Theta$ in \cite{beale_1}.  

In light of the decay of $u$ in \eqref{secc1} of \cite{beale_2}, we may not expect a global well-posedness of the problem in horizontally infinite domains without surface tension, \eqref{ns_euler}, with the solution satisfying \eqref{vdecay}. Hataya and Kawashima \cite{HK}  announced that  for the problem \eqref{geometric}  with a flat bottom  in $3D$,  if $u_0 \in H^{r-1}(\Omega)$ and $\eta_0 \in H^{r-1/2}(\Sigma)  $ for $r\in (5,11/2)$ and $\eta_0\in  L^1(\Sigma)$ are sufficiently small, then there exists a  unique global solution satisfying
\begin{equation}\label{secc4}
 (1+t)^2 \ns{u(t)}_{H^1(\Omega)}  +   \sum_{j=0}^1 (1+t)^{j+1} \ns{D^j \eta (t)}_{L^2(\Sigma)}  < \infty ,
\end{equation} 
but they provides only a terse sketch of their proposed proof and the full details are not available  in the literature to date.  Guo and Tice \cite{GT_inf} proved that for the problem \eqref{geometric} with a flat bottom  in $3D$, if $u_0 \in H^{20}(\Omega)$, $\eta_0 \in H^{20+1/2}(\Sigma)$, $\i_\lambda u_0\in L^2(\Omega)$ and  $\i_\lambda\eta_0\in L^2(\Sigma)$ for $0<\lambda<1$ ($\i_\lambda$ is the Riesz potential in the horizontal space) are sufficiently small, then there exists a  unique global solution such that 
\begin{equation}\label{intro_inf_gwp_03} 
 (1+t)^{1+\lambda} \ns{u(t)}_{H^2(\Omega)}  +   \sum_{j=0}^2 (1+t)^{j+\lambda}  \ns{D^j \eta (t)}_{L^2(\Sigma)}  < \infty.
\end{equation}
We remark that it was pointed out in \cite{GT_inf} that the requirement of $\lambda>0$ is necessary for their argument.

Note that both \cite{GT_inf} and  \cite{HK} proved the global well-posedness of \eqref{geometric}  in 3D  by requiring the certain low frequency assumption of the initial data, while the global well-poseness in $2D$ was left open. The main purpose of this paper is to show the global well-posedness of \eqref{geometric}  in both $2D$ and $3D$, without any low frequency assumption of the initial data. This gives a closer answer to the question in the nondecay theorem of Beale \cite{beale_1}. It should be pointed out  that the global well-posedness of the problem with a curved bottom is still open, and the key point will be how to deduce the decay of the solution.

\section{Main results}
\subsection{Statement of the results}\label{sec2.1}
 We will work in a high-regularity context, essentially with regularity up to $2N$ temporal derivatives for an integer $N\ge 5$.  This requires us to use $u_0$ and $\eta_0$, by using the equations \eqref{geometric}, to construct the initial data $\dt^j u(0)$ and $\dt^j \eta(0)$ for $j=1,\dotsc,2N$ and $\dt^j p(0)$ for $j = 0,\dotsc, 2N-1$.  These data must then satisfy various conditions, which in turn require $u_0$ and $\eta_0$ to satisfy $2N$ compatibility conditions.  We refer the reader to \cite{GT_lwp} for the construction of those initial data and the precise description of the $2N$ compatibility conditions.

We write $H^k(\Omega)$ with $k\ge 0$ and $H^s(\Sigma)$ with $s \in \Rn{}$ for the usual Sobolev spaces, with norms denoted by $\norm{\cdot}_{k}$ and $\abs{\cdot}_{s}$, respectively.
For a vector $v\in \mathbb{R}^d$ for  $d=2,3$, we write $v=(v_h,v_d)$ for $v_h$ the horizontal component of $v$ and $v_d$ the vertical component.
 We   write $ D  f$ for the horizontal gradient of $f$, while $\nab f$  denotes the usual full gradient. Let $d=2,3$ and $N\ge 5$. We define the high-order energy as
\begin{align}\label{p_energy_def1}
 \fe{2N} : =  \sum_{j=0}^{2N}  \ns{\dt^j u}_{4N-2j}    + \sum_{j=0}^{2N-1} \ns{\dt^j p}_{4N-2j-1} +\as{\eta}_{4N-1 } +\sum_{j=1}^{2N}  \as{\dt^j \eta}_{4N-2j}
\end{align}
and the high-order dissipation rate as
\begin{align}\label{p_dissipation_def1}
 \fd{2N}  := & \ns{ u}_{4N}+\sum_{j=1}^{2N} \ns{\dt^j u}_{4N-2j+1} +  \ns{\nabla p}_{4N-2}+ \sum_{j=1}^{2N-1} \ns{\dt^j p}_{4N-2j} \nonumber\\
&+  \as{ D\eta}_{4N-5/2}+ \as{\dt \eta}_{4N-3/2} + \sum_{j=2}^{2N+1} \as{\dt^j \eta}_{4N-2j+5/2}.
\end{align}
We define
 \begin{align}\label{pdd_dissipation_def1}
 \sem{2N}  :=  \as{\eta}_{4N}
 \end{align}
 and
\begin{align}\label{pdd_dissipation_def2}
 \sdm{2N}  :=  \ns{ u}_{4N+1} +  \ns{\nabla p}_{4N-1}+   \as{ D\eta}_{4N-3/2}+   \as{ \dt  \eta}_{4N-1/2}.
 \end{align}
We also define
\begin{align}\label{pdd_dissipation_def3}
 \f :=  \as{\eta}_{4N+1/2}  .
\end{align}
We define the low-order energy as
\begin{align}\label{i_energy_min_2}
 \senn   = & \ns{D u_h}_{2(N+2)-1} + \ns{  u_d}_{2(N+2)} + \ns{\nabla^3 u }_{2(N+2)-3} +\sum_{j=1}^{N+2}  \ns{\dt^j u}_{2(N+2)-2j}\nonumber
 \\& + \ns{\nabla^2 p}_{2(N+2)-3}  + \ns{\p_d p}_{2(N+2)-2}  + \sum_{j=1}^{N+1} \ns{\dt^j p}_{2(N+2)-2j-1}\nonumber
   \\& +\as{D^2 \eta}_{2(N+2)-2} +\sum_{j=1}^{N+2}  \as{\dt^j \eta}_{2(N+2)-2j}.
\end{align}
Here the subscript ``$2$" basically refers to the ``minimal derivative" count $2$ of $\eta$.

Finally, we define
\begin{align}\label{ggdef}
\mathcal{G}_{2N}  (t) := &\sup_{0 \le r \le t} \fe{2N} (r) + \int_0^t \fd{2N} (r) dr+  \sup_{0 \le r \le t} \frac{\sem{2N}(r)}{(1+r)^{\vartheta}}+\int_0^t\frac{ \sem{2N} (r) }{(1+r)^{1+\vartheta }}dr  \nonumber
\\&+\int_0^t\frac{ \sdm{2N} (r) }{(1+r)^{\vartheta+\kappa_d}}dr +   \sup_{0 \le r \le t} \frac{\f (r)}{(1+r)^{1+\vartheta}}+\int_0^t\frac{ \f(r) }{(1+r)^{2+\vartheta }}dr\nonumber\\&+ \sup_{0 \le r \le t} (1+r)^{2} \senn (r),
\end{align}
where  $\vartheta>0$ is any sufficiently small constant, and when $d=3$, $\kappa_3>0$ is any sufficiently small constant  and when $d=2$, $\kappa_2=1/2$. Then the main result of this paper is stated as follows.

 \begin{thm}\label{intro_inf_gwp}
 Let $d=2,3$ and $N\ge 5$.
Suppose the initial data $(u_0,\eta_0)$ satisfy the necessary compatibility conditions of the local well-posedness of \eqref{geometric}. There exists an $\varepsilon_0 >0$ so that if $\fe{2N}(0) + \f(0) \le \varepsilon_0$, then there exists a unique solution $(u,p,\eta)$ to \eqref{geometric} on the interval $[0,\infty)$ that achieves the initial data.  The solution obeys the estimate
\begin{equation}\label{intro_inf_gwp_01}
 \g(\infty) \ls  \fe{2N}(0) + \f(0).
\end{equation}
In particular,  we have
\begin{align}\label{intro_inf_gwp_02}
& (1+t)^{2-\kappa_d} \ns{u_h(t)}_{C^2(\bar\Omega)}     +(1+t)^{2} \(\ns{Du_h(t)}_{C^1(\bar\Omega)} +\as{\p_d u_h(t)}_{C(\Sigma)} +\ns{u_d(t)}_{C^2(\bar\Omega)} \)\nonumber
\\&\quad\ls  \fe{2N}(0) + \f(0).
\end{align}
\end{thm}
\begin{remark}
The boundedness of $ \g(\infty) $ in Theorem \ref{intro_inf_gwp}, \eqref{intro_inf_gwp_01}, yields the decay estimate of $\senn(t)$.
Note that the decay  of the velocity  is anisotropic,  in terms of the horizontal or vertical components, and the horizontal or vertical spatial derivatives (on $\Omega$ or $\Sigma$) of the horizontal component. Such anisotropic decay of the velocity is essential in the derivation of the estimate \eqref{intro_inf_gwp_01}. We remark that without the low frequency assumption of the initial data the decay rate $(1+t)^{-2}$ of $\senn(t)$, mainly referring to the term $\abs{D^2\eta}_0^2$, is optimal, see \cite{beale_nishida,GT_inf}.
\end{remark}

\subsection{Strategy of the proof}\label{strategy sec}
The proof of Theorem \ref{intro_inf_gwp} is a revised one of Theorem 1.3 in \cite{GT_inf}, and the main part is to prove {\it a priori} \eqref{intro_inf_gwp_01}. However, a crucial new observation and several new ideas, which will  be explained in details below, yield our improvement of Theorem \ref{intro_inf_gwp} compared to Theorem 1.3 in \cite{GT_inf}; especially, we do not need the low frequency assumption of the initial data as required in \cite{GT_inf}, and our results cover both the $3D$ and $2D$ cases.  

We first briefly sketch the proof of Theorem 1.3 in \cite{GT_inf} for the $3D$ case. Note that $\se{2N}:=\fe{2N}+\sem{2N}$ and $\sd{2N}:=\fd{2N}+\sdm{2N}$ are the high-order energy and dissipation rate used in \cite{GT_inf} when $\lambda=0$ therein.  Basing on the natural energy structure of \eqref{geometric},
\begin{align} \label{intro0}
\hal  \frac{d}{dt} \left(\int_\Omega  J \abs{u}^2+\int_\Sigma    \abs{\eta}^2 \right)+\hal\int_\Omega J \abs{\mathbb{D}_\a u}^2= 0,
\end{align}
and making full use of the structure of the problem, \cite{GT_inf} proved that for $\mathcal{E}_{2N}$ is small,
 \begin{equation}\label{pp}
\mathcal{E}_{2N}(t)+\int_0^t \mathcal{D}_{2N}(r)\,dr \lesssim
 \mathcal{E}_{2N}(0)  
+ \int_0^t \mathcal{K}(r)\mathcal{F}_{2N}(r)\,dr,
\end{equation}
where $\mathcal{K}=\ns{ \nabla u}_{C(\bar\Omega)}+\as{D  \nabla u}_{C(\Sigma)}$. The difficulty is  then to  control the right hand side of \eqref{pp}, and the only way to estimate $\f$  is through the fourth transport equation for $\eta$ in \eqref{geometric}; \cite{GT_inf} derived
\begin{equation} \label{pp11}
   \f(t) \ls \exp\(C\int_0^t \sqrt{\mathcal{K}(r)}\,dr\)\[
\f(0) +   t \int_0^t \sd{2N}(r)\,dr\].
\end{equation}
Hence to close \eqref{pp} and \eqref{pp11}, we would see twice the necessity of showing $\mathcal{K}(t)\ls (1+t)^{-2-\gamma}$ for some $\gamma>0$. By assuming that $\i_\lambda u_0\in L^2(\Omega)$ and  $\i_\lambda\eta_0\in L^2(\Sigma)$ for $0<\lambda<1$ are sufficiently small,  \cite{GT_inf} proved the decay estimate of the low-order energy $ \mathcal{E}_{N+2,2}$ (see \eqref{defe2} for the definition),
\begin{equation}   \label{intro01}
 \mathcal{E}_{N+2,2}(t)  \ls  (\mathcal{E}_{2N}(0) +\f(0))(1+t)^{-2-\lambda } .
\end{equation}
It follows by the interpolation estimates (similarly as in Lemma \ref{klem}) and \eqref{intro01} that
\beq\label{intro02}
\mathcal{K} \ls \se{2N}^{\kappa_3/2} \mathcal{E}_{N+2,2}^{1-\kappa_3/2}\ls  (1+t)^{-(2+\lambda) (1-\kappa_3/2)}.
\eeq   
Consequently, \cite{GT_inf} can close the estimates \eqref{pp}--\eqref{pp11} by requiring $\lambda>0$ in $3D$. We remark that the energy estimates for the  $2D$ case can not be closed along the same way as above since $(2+\lambda)(1-\kappa_2/2)<15/8$ for $\kappa_2=1/2$ and $0<\lambda<1/2$ ($\lambda<1/2$ is restricted for the nonlinear estimates involving $\i_\lambda$ in dimension $1$).  Hence, the global well-poseness in $2D$ was left open.

Our main goals of this paper are to remove the assumption of $\lambda>0$ by \cite{GT_inf}  in $3D$ and to cover the $2D$ case.  This requires us to revise all the estimates \eqref{pp}--\eqref{intro01}. Our starting point is that we can improve the decay estimate \eqref{intro01} by replacing $ \mathcal{E}_{N+2,2}$ with the stronger low-order energy $\senn$ so that for $\lambda=0$,
\begin{equation}  \label{intro1}
(1+t)^{2}  \senn(t)  \ls  \fe{2N}(0) +\f(0).
\end{equation}
Such improved decay estimate is essential for our improvement compared to \cite{GT_inf}.
Secondly, we refine from the derivation of the estimate \eqref{pp11} that
\begin{equation} \label{intro2}
\dt   \f \ls \abs{u}_{4N+1/2}\sqrt{\f}+\abs{Du_h}_{C(\Sigma)}\f.
\end{equation}
The subtle points here are that we need only the boundary control of  $\as{u}_{4N+1/2}$ rather than $\sd{2N}$ and that $ \abs{Du_h}_{C(\Sigma)}^2\ls { \senn}$. So if we can prove that  
\begin{equation}  \label{intro001}
\int_0^t\frac{\as{u(r)}_{4N+1/2}}{(1+r)^{\vartheta}}dr\ls \fe{2N}(0) +\f(0),\  \vartheta\ge0,
\end{equation}
then a time weighted argument on \eqref{intro2} yields 
\begin{equation}\label{intro3}
\frac{\f(t)}{(1+t)^{1+\vartheta}}+\int_0^t\frac{\f(r)}{(1+r)^{2+\vartheta}}dr\ls \fe{2N}(0) +\f(0),\  \vartheta\ge0.
\end{equation}
Apparently,  the estimates \eqref{intro1} and \eqref{intro3} are not sufficient for closing the estimate \eqref{pp},  and we need to revise  \eqref{pp}. We recall  from \cite{GT_inf} that the term $\mathcal{K}\f$ in \eqref{pp} results from the nonlinear estimates of the viscous stress tensor term when estimating the highest order spatial derivatives of the solution. So we will split the estimates \eqref{pp} by singling out $\sem{2N}$ from $\se{2N}$ and $\sdm{2N}$ from $\sd{2N}$. In this splitting, we can first prove 
 \begin{equation}\label{intro4}
\fe{2N}(t)+\int_0^t \fd{2N}(r)\,dr \lesssim
 \fe{2N}(0) +\f(0).
\end{equation}
The  crucial part is then to control $\sem{2N}$ and $\sdm{2N}$. Applying the highest order horizontal spatial derivatives $\p^\al$ for $\al\in \mathbb{N}^{d-1}$ with $\abs{\al}=4N$ to \eqref{geometric}, we find
\begin{align} \label{intro5}
&\hal  \frac{d}{dt} \left(\int_\Omega  J \abs{\p^{\al}u}^2 \right)+ \int_\Omega J \(\hal \p^{\al}\(\mathbb{D}_\a u\) : \mathbb{D}_\a\p^{\al}u- \p^\al p  \diverge_\a  \p^\al u\)
\nonumber
\\&\quad=\int_\Sigma \p^\al (pI-\mathbb{D}_\a u)\n   \cdot \p^{\al}u+\cdots.
\end{align}
Note that when $\p^\al$ hints $\a$ leaded to the appearance of the term $\mathcal{K}\f$ in \eqref{pp}, and ``$+\cdots$" denotes the terms that can be controlled other than $\mathcal{K}\f$. Then our idea is to use a crucial cancelation observed by Alinhac \cite{Alinhac}. More precisely, the direct computation yields
\beq\label{intro6}
\p^\alpha \D_{i} u =\D_{i}(\p^\alpha u-\D_{d}u  \p^\alpha \varphi)+\D_{d}\D_{i}u \p^\alpha \varphi+\cdots,
\eeq
which implies that the highest order term of $\eta$ will be canceled when we use the good unknown  $U^\al=\p^\al u-\p_d ^\a u\p^\al \varphi$. Thus considering the equations of $U^\al$ in $\Omega$, instead of \eqref{intro5}, we get
\begin{align} \label{intro7}
\hal  \frac{d}{dt} \left(\int_\Omega  J \abs{U^\al}^2 \right)+\hal\int_\Omega J \abs{\mathbb{D}_\a U^\al}^2=  \int_\Sigma (\p^\al pI-\mathbb{D}_\a U^\al)\n   \cdot U^\al+\cdots.
\end{align}
This means that we have canceled the term $\mathcal{K}\f$ in the bulk.
By estimating the right hand side of \eqref{intro7} by using the boundary conditions in \eqref{geometric},  we obtain
\begin{equation}\label{intro8}
\hal  \frac{d}{dt} \left(\int_\Omega  J \abs{U^\al}^2+\int_\Sigma    \abs{\p^\al\eta}^2 \right)+\hal\int_\Omega J \abs{\mathbb{D}_\a U^\al}^2  \ls  \abs{\nabla u}_{C^1(\Sigma)}^2\f+\abs{\nabla u}_{C^1(\Sigma)}\sem{2N}+\cdots
.
\end{equation}
Note that the term $\abs{\nabla u}_{C^1(\Sigma)}^2\f$ in \eqref{intro8} stems from  using the third equation in \eqref{geometric}, which can not be canceled by using Alinhac good unknowns, and $\abs{\nabla u}_{C^1(\Sigma)}\sem{2N}$ results from using the fourth equation in \eqref{geometric}.
However, the crucial point here is that by using the horizontal component of the third equation in \eqref{geometric}, we can show that
$\abs{\nabla u}_{C^1(\Sigma)}^2\ls { \senn}$.
Then a time weighted argument on \eqref{intro8}, together with \eqref{intro0}, yields
\begin{equation} \label{intro9}
 \frac{ \sem{2N}(t)}{(1+t)^{\vartheta}} +\int_0^t  \frac{ \sem{2N}(r)}{(1+r)^{1+\vartheta}} dr +\int_0^t\frac{\ns{ U^{4N}(r)}_{1}}{(1+r)^{\vartheta}}dr\ls
 \fe{2N}(0) +\f(0) , \ \vartheta>0,
\end{equation}
where we denote $\ns{ U^{4N}}_{1}:=\ns{u}_0+\sum_{\al\in \mathbb{N}^{d-1},\abs{\al}=4N}\ns{ U^\al}_{1}$. Another advantage of the good unknown is that, by the definition of $U^\al$,
 \beq
 \abs{u}_{4N+1/2}^2\ls\ns{ U^{4N}}_{1}+\abs{\p_du}_{C^1(\Sigma)}^2\f\ls\ns{ U^{4N}}_{1}+\senn\f.
\eeq
Thus, \eqref{intro9} implies the validity of \eqref{intro001} for $\vartheta>0$ and hence \eqref{intro3} for $\vartheta>0$. Then we see reasonably that not assuming $\lambda>0$ allows for the faster growth of $\f$ in time. 
Finally, to control $\sdm{2N}$ it still involves $\mathcal{K}\f$, and by the interpolation estimate \eqref{kes} of $\mathcal{K}$ in Lemma \ref{klem}, $
\mathcal{K} \ls  (1+t)^{-2 (1-\kappa_d/2)}$, we can show
\begin{equation}\label{intro10}
\int_0^t\frac{\sdm{2N}(r)}{(1+r)^{\vartheta+\kappa_d}}dr \lesssim
 \fe{2N}(0) +\f(0).
\end{equation}
We remark that in the derivation of the estimates \eqref{intro1},  \eqref{intro3}, \eqref{intro4}, \eqref{intro9} and \eqref{intro10}, certain powers of greater than 1 of $\g(t)$ need to be added on the right hand sides of these estimates. Consequently, summing over these estimates, the a priori estimate \eqref{intro_inf_gwp_01} is then closed by assuming that $\fe{2N}(0) +\f(0)$ is sufficiently small. The proof of Theorem \ref{intro_inf_gwp} can be thus completed by a continuity argument by combining the local existence theory in Guo and Tice \cite{GT_lwp} and our a priori estimates. We remark that although the strategy  is carried out unifiedly for both $3D$ and $2D$ cases,  the analysis in $2D$ is much more involved.

\subsection{Notation}\label{nota}

We write $\mathbb{N} = \{ 0,1,2,\dotsc\}$ for the collection of non-negative integers.  When using space-time differential multi-indices, we write $\mathbb{N}^{1+m} = \{ \alpha = (\alpha_0,\alpha_1,\dotsc,\alpha_m) \}$ to emphasize that the $0-$index term is related to temporal derivatives.  For just spatial derivatives we write $\mathbb{N}^m$.  For $\alpha \in \mathbb{N}^{1+m}$ we write $\pa = \dt^{\alpha_0} \p_1^{\alpha_1}\cdots \p_m^{\alpha_m}.$ We define the parabolic counting of such multi-indices by writing $\abs{\alpha} = 2 \alpha_0 + \alpha_1 + \cdots + \alpha_m.$ For $\al\in \mathbb{N}^d$, we write $\al=(\al_h,\al_d).$

For a given norm $\norm{\cdot}$ and  integers $k\ge m\ge 0$, we
introduce the following notation for sums of spatial derivatives:
\begin{equation}\label{deco1}
 \norm{{D}_m^k f}^2 := \sum_{\substack{\al \in \mathbb{N}^{d-1} \\ m\le \abs{ \alpha}\le k} } \norm{\pa  f}^2 \text{ and }
\norm{\nab_m^k f}^2 := \sum_{\substack{\alpha \in \mathbb{N}^{1+(d-1)} \\  m\le \abs{\alpha}\le k} } \norm{\pa  f}^2.
\end{equation}
For space-time derivatives we add bars to our notation:
\begin{equation}\label{deco2}
 \norm{\bar{D}_{m}^{k}  f}^2 := \sum_{\substack{\alpha \in \mathbb{N}^{1+(d-1)} \\  m\le \abs{\alpha}\le k} } \norm{\pa  f}^2 \text{ and }
\norm{\bar{\nab}_m^k f}^2 := \sum_{\substack{\alpha \in \mathbb{N}^{1+d} \\ m\le  \abs{\alpha}\le k} } \norm{\pa  f}^2.
\end{equation}
When $k=m \ge 0$ we will write
\begin{equation}\label{deco3}
 \norm{D^k f}^2 = \norm{D_k^k f}^2, \norm{\nab^k f}^2 =\norm{\nab_k^k f}^2, \norm{\bar{D}^k f}^2 = \norm{\bar{D}_k^k f}^2, \norm{\bar{\nab}^k f}^2 =\norm{\bar{\nab}_k^k f}^2.
\end{equation}

 We
employ the Einstein convention of summing over repeated indices for vector and tensor operations. Throughout the paper we assume that $N\ge 5$ is an integer.
$C>0$   denotes a generic constant that can depend on the parameters of the problem, $d=2,3$, $ N,\ \vartheta$ and $\kappa_d$, but does
not depend on the data, etc. We refer to such constants as ``universal''.   They are allowed to change from line to line.  We
employ the notation $A_1 \lesssim A_2$ to mean that $A_1 \le C A_2$ for a
universal constant $C>0$. To avoid the constants in various time differential inequalities, we employ the following two conventions:
\beq 
\dt A_1+A_2\ls A_3 \text{ means }\dt \widetilde A_1+A_2\ls A_3 \text{ for any }A_1\ls \widetilde A_1\ls A_1
\eeq
and
\beq 
\dt \(A_1+A_2\)+A_3\ls A_4 \text{ means }\dt \(C_1A_1+C_2A_2\)+A_3\ls A_4 \text{ for constants }C_1,C_2>0.
\eeq
We omit the differential elements of the integrals over $\Omega$ and $\Sigma$, and also sometimes the time differential elements.

\section{Preliminaries}
We will assume throughout the rest of the paper that the solutions are given on the interval $[0,T]$ and obey the a priori assumption
\beq\label{apriori_1}
\g(t)\le \delta,\quad\forall t\in[0,T]
\eeq
for an integer $N\ge 3$ and a sufficiently small constant $\delta>0.$ This implies in particular that
\beq\label{apriori_2}
\hal\le J\le \frac{3}{2},\quad\forall (t,x)\in[0,T]\times \bar\Omega.
\eeq
\eqref{apriori_1} and \eqref{apriori_2} will be used frequently, without mentioning explicitly.

\subsection{Energy functionals}
Below we define the energy functionals used in our analysis.
We recall the definitions of  $\fe{2N},\  \fd{2N},\  \sem{2N},\  \sdm{2N},\ \f$, $\senn$ and $\g$ from \eqref{p_energy_def1}--\eqref{ggdef} in Section \ref{sec2.1}, respectively. We define the low-order dissipation rate by
\begin{align}\label{p_ldissipation_def}
 \sdnn   = & \ns{D^2 u_h}_{2(N+2)-1} + \ns{ D u_d}_{2(N+2)}+\ns{ \nabla^4 u }_{2(N+2)-3}
 +\sum_{j=1}^{N+2}  \ns{\dt^j u}_{2(N+2)-2j+1} \nonumber\\& + \ns{\nabla^3 p}_{2(N+2)-3}+ \ns{D\p_d p}_{2(N+2)-2}
+ \ns{\nabla \dt p}_{2(N+2)-3} +\sum_{j=2}^{N+1} \ns{\dt^j p}_{2(N+2)-2j} \nonumber\\& +\as{D^3 \eta}_{2(N+2)-7/2}
+ \as{D\dt \eta}_{2(N+2)-3/2} + \sum_{j=2}^{N+3} \as{\dt^j \eta}_{2(N+2)-2j+5/2}.
\end{align}
Recall that we employ the derivative conventions  \eqref{deco1}--\eqref{deco3} from Section \ref{nota}.
We define the high-order tangential energy by
\begin{equation}\label{i_horizontal_energy}
 \feb{2N} := \ns{ {D}_{0}^{4N-1}   u}_{0} + \ns{\bar{D}_{0}^{4N-2} \dt u}_{0}+ \as{ {D}_{0}^{4N-1}   \eta}_{0} + \as{\bar{D}_{0}^{4N-2}\dt \eta}_{0} 
\end{equation}
and the corresponding tangential dissipation rate by
\begin{equation}\label{i_horizontal_dissipation}
 \fdb{2N} :=  \ns{ {D}_{0}^{4N-1}   u}_{1} + \ns{\bar{D}_{0}^{4N-2} \dt u}_{1}.
\end{equation}
The low-order tangential energy is  
\begin{equation}\label{i_horizontal_energy1}
 \feb{N+2,2} :=  \ns{\bar{D}_{2}^{2(N+2)}  u}_{0}+ \as{\bar{D}_{2}^{2(N+2)}   \eta}_{0} ,
\end{equation}
and the corresponding tangential dissipation  rate is
\begin{equation}\label{i_horizontal_dissipation1}
 \fdb{N+2,2} :=  \ns{\bar{D}_{2}^{2(N+2)}  u}_{1}.
\end{equation}
We also define two special quantities
\beq\label{kdef}
\k:=\ns{  u}_{C^1(\bar\Omega)}+\as{ \nabla^2 u}_{C^1(\Sigma)}
\eeq
and 
\beq\label{bkdef}
\bar\k:=\as{ \nabla u}_{C^1(\Sigma)}.
\eeq
Note that $\bar\k\ls\k$.

We have the following lemma that constrains $N$.

\begin{lem}\label{i_N_constraint}
If $N \ge 3$, then we have that  $\fe{N+2,2} \ls \fe{2N}$  and $\fd{N+2,2} \ls \fe{2N}.$
\end{lem}
\begin{proof}
The proof follows by simply comparing the definitions of these terms.
\end{proof}

For the convenience of comparing our estimates with those of \cite{GT_inf}, we also recall the energy functionals used in \cite{GT_inf}. First,  $\se{2N}:=\fe{2N}+\sem{2N}$ and $\sd{2N}:=\fd{2N}+\sdm{2N}$ are the high-order energy and dissipation rate used in \cite{GT_inf} when $\lambda=0$ therein. Next, \cite{GT_inf} used the following low-order energy
\begin{align} \label{defe2}
\se{N+2,2}  = & \ns{D^{2(N+2)}_2 u}_{0} +\ns{\nabla^3 u }_{2(N+2)-3} +\sum_{j=1}^{N+2}  \ns{\dt^j u}_{2(N+2)-2j}+ \ns{\nabla^2 p}_{2(N+2)-3}   \nonumber
 \\& + \sum_{j=1}^{N+1} \ns{\dt^j p}_{2(N+2)-2j-1}+\as{D^2 \eta}_{2(N+2)-2} +\sum_{j=1}^{N+2}  \as{\dt^j \eta}_{2(N+2)-2j}
\end{align}
and the corresponding dissipation rate
\begin{align} 
\sd{N+2,2}  = & \ns{D^{2(N+2)}_2 u}_{1}  +\ns{ \nabla^4 u }_{2(N+2)-3}
 +\sum_{j=1}^{N+2}  \ns{\dt^j u}_{2(N+2)-2j+1} \nonumber\\& + \ns{\nabla^3 p}_{2(N+2)-3}+  \ns{\nabla \dt p}_{2(N+2)-3} +\sum_{j=2}^{N+1} \ns{\dt^j p}_{2(N+2)-2j} \nonumber\\& +\as{D^3 \eta}_{2(N+2)-7/2}
+ \as{D\dt \eta}_{2(N+2)-3/2} + \sum_{j=2}^{N+3} \as{\dt^j \eta}_{2(N+2)-2j+5/2}.
\end{align}
\cite{GT_inf} also used $\f$, but they used
\beq\label{kkkdef}
\mathcal{K}:=\ns{ u}_{C^1(\bar\Omega)}+\as{D  \nabla u}_{C(\Sigma)}.
\eeq
Note that $\bar\k\ls\mathcal{K}\ls\k$.

\subsection{Perturbed linear form}
In order to use the linear structure of the equations \eqref{geometric}, we will write it as the a perturbation of the linearized equations:
\begin{equation}\label{linear_perturbed}
 \begin{cases}
  \dt u + \nab p - \Delta u = G^1 & \text{in }\Omega \\
  \diverge{u} = G^2 & \text{in }\Omega \\
  (p I_d - \sg u  )e_d  =   \eta  e_d  +G^3 & \text{on }\Sigma \\
  \dt \eta  = u_d +G^4 & \text{on } \Sigma \\
  u =0 & \text{on }\Sigma_b.
 \end{cases}
\end{equation}
Here we have written the vector  $G^1 = G^{1,1} + G^{1,2} + G^{1,3} + G^{1,4} + G^{1,5}$ for
\begin{align}\label{Gi_def_start}
 &G^{1,1}_i = (\delta_{ij} - \mathcal{A}_{ij} )\p_j p,
 \\
 &G^{1,2}_i = u_j\mathcal{A}_{jk} \p_k u_i,
\\
& G^{1,3}_i = [ K^2(1+|D\varphi|^2) - 1]\p_{d}^2u_i - 2KD\varphi\! \cdot\! D \p_{d}u_i ,
\\
& G^{1,4}_i
= [  -K^3(1+|D\varphi|^2) \p_d  J + K^2D\varphi\! \cdot\!  (D J  + D\p_d \varphi)   - K D \! \cdot\! D\varphi\p_d  u_i,
\\
 &G^{1,5}_i =K \dt \varphi  \p_d  u_i.
\end{align}
$G^2$ is the function
\begin{equation}
 G^2=  KD\varphi\! \cdot\! \p_{d}u_h + (1-K)\p_d  u_d ,
\end{equation}
and $G^3$ is the vector defined by that for $d=3$,
\begin{align}\label{G3_def}
G^3 :=&  \p_1 \eta
\begin{pmatrix}
 p-\eta -2(\p_1 u_1 -AK \p_3  u_1  ) \\
 -\p_2 u_1 - \p_1 u_2  + BK \p_3  u_1 + AK \p_3  u_2 \\
 -\p_1 u_3  - K \p_3  u_1 + AK \p_3  u_3
\end{pmatrix}\nonumber
\\ &+
\p_2 \eta
\begin{pmatrix}
  -\p_2 u_1 - \p_1 u_2  + BK \p_d  u_1 + AK \p_3  u_2  \\
  p-\eta -2(\p_2 u_2 -BK \p_3  u_2  )  \\
 -\p_2 u_3  - K \p_3  u_2 + BK \p_3 u_3 
\end{pmatrix}
+
\begin{pmatrix}
  (K-1) \p_3  u_1  +AK \p_3  u_3  \\
  (K-1) \p_3  u_2  +BK \p_3  u_3   \\
  2(K-1)\p_3  u_3 
\end{pmatrix} 
\end{align}
and that for  $d=2$,
\beq\label{G32_def}
G^3 :=  \p_1 \eta
\begin{pmatrix}
 p-\eta -2(\p_1 u_1 -AK \p_2  u_1  ) \\
  -\p_1 u_2  - K \p_d  u_1 + AK \p_2  u_2
\end{pmatrix} 
+
\begin{pmatrix}
  (K-1) \p_2  u_1  +AK \p_2  u_2  \\
  2(K-1)\p_2  u_2 
\end{pmatrix},
\eeq
where $A=\p_1\eta$ and $B=\p_2\eta$.
Note that, according to \eqref{linear_perturbed},
\begin{align}\label{G3_alternate}
 p-\eta =\p_d  u_d  + G^3_d   =- D \eta\!\cdot\! ( D u_d+K \p_d  u_h - K D\eta\p_d  u_d ).
\end{align}
Finally,
\begin{equation}\label{Gi_def_end}
 G^4 = -D \eta\!\cdot\! u_h.
\end{equation}

\subsection{Interpolation estimates}
The fact that $\senn $ and $\sdnn $ have a minimal count of derivatives creates numerous problems when we try to estimate terms with fewer derivatives in terms of $\senn $ and $\sdnn $.  Our way around this is to interpolate between $\senn $ (or $\sdnn $) and $\fe{2N}$. We will prove various interpolation inequalities of the form
\begin{equation}\label{interp_form}
 \ns{X} \ls (\senn )^\theta (\fe{2N})^{1-\theta} \text{ and } \ns{X} \ls (\sdnn )^\theta (\fe{2N})^{1-\theta},
\end{equation}
where $\theta \in [0,1]$, $X$ is some quantity, and $\norm{\cdot}$ is some norm (usually either $L^2$ or $L^\infty$).
In the interest of brevity, we will record these estimates in tables that only list the value of $\theta$ in the estimate. For example,
\beq\begin{array}[t]{| c | c  | c |}
\hline
 L^2   \mid  \senn   &    2D  \\ \hline
D\eta & 1/2 \\ \hline
\end{array}
\quad
\begin{array}[t]{| c | c  | c |}
\hline
 L^\infty   \mid \sdnn      &     3D \\ \hline
\nabla\bar\eta & 2/3\\ \hline
\end{array}\nonumber
\eeq
We understand this to mean that
$$\as{D\eta}_0\ls (\senn )^{1/2} (\fe{2N})^{1/2} \text{ in } 2D, \text{ and }\ns{\nabla\bar\eta}_{L^\infty}\ls (\sdnn )^{1/2} (\fe{2N})^{1/2}\text{ in } 3D.$$
 
We record the interpolation estimates in the following lemma, where the norms for $\bar\eta,u,p, G^1$ and $G^2$ are on $\Omega$ and the norms for $\eta,G^3$ and $G^4$ are on $\Sigma$. In the below $r>0$ will denote for any small constant. 
\begin{lem}\label{le inter}

Let $u,p,\eta$ be the solution of \eqref{linear_perturbed} and $G^i$ be defined in \eqref{Gi_def_start}--\eqref{Gi_def_end}.

\begin{itemize}

\item[(1)] The following tables encode the powers in the $L^2$ and $L^\infty$ interpolation estimates for the solution and their derivatives in terms of $\senn$:
\beq \label{t1}
\begin{split}
\begin{array}[t]{| c | c  | c |}
\hline
 L^2   \mid  \senn   &    2D & 3D \\ \hline
\eta, \bar\eta & 0 & 0\\ \hline
D\eta, \nabla\bar\eta & 1/2 & 1/2\\ \hline
Dp & 1/2 & 1/2\\ \hline
 u_h,\p_d u_h, \p_d^2 u_h   & 1/2& 1/2   \\ \hline
\end{array}
\quad
\begin{array}[t]{| c | c  | c |}
\hline
 L^\infty   \mid \senn      &     2D &3D \\ \hline
\eta, \bar\eta & 1/4 & 1/2\\ \hline
D\eta, \nabla\bar\eta & 3/4 & 1/(1+r)\\ \hline
Dp & 3/4 & 1/(1+r)\\ \hline
 u_h,\p_d u_h, \p_d^2 u_h & 3/4 & 1/(1+r) \\ \hline
\end{array}
\end{split}
\eeq

 The following tables encode the powers in the $L^2$ and $L^\infty$ interpolation estimates for the solution and their derivatives in terms of $\sdnn$:
\beq\label{t2}
\begin{split}
\begin{array}[t]{| c | c  | c |}
\hline
 L^2   \mid \sdnn  & 2D &3D      \\ \hline
\eta, \bar\eta & 0 & 0\\ \hline
D\eta, \nabla\bar\eta & 1/3 & 1/3\\ \hline
D^2\eta, \nabla^2\bar\eta, 
\dt\eta, \dt\bar\eta &2/3&2/3\\ \hline
Dp & 1/3 & 1/3\\ \hline
D^2p  ,
 \p_dp, \p_d^2 p ,
  \dt p  &2/3&2/3\\  \hline
 u_h,\p_d u_h, \p_d^2 u_h   & 1/3 & 1/3  \\ \hline
 Du_h,D\p_d u_h, D\p_d^2 u_h   & 2/3 & 2/3  \\ \hline
\p_d^3 u_h   & 5/6 & 1     \\ \hline
 u_d,\p_d u_d, \p_d^2 u_d, \p_d^3 u_d    & 2/3  & 2/3  \\ \hline
\end{array}
 \quad
\begin{array}[t]{| c | c  | c |}
\hline
 L^\infty   \mid\sdnn  & 2D&3D      \\ \hline
\eta, \bar\eta & 1/6 & 1/3\\ \hline
D\eta, \nabla\bar\eta & 1/2 & 2/3\\ \hline
D^2\eta, \nabla^2\bar\eta ,
\dt\eta, \dt\bar\eta&5/6 & 1/(1+r)\\ \hline
Dp & 1/2 & 2/3\\ \hline
D^2p  ,
 \p_dp, \p_d^2 p,
\dt p  &5/6& 1/(1+r)\\  \hline
 u_h,\p_d u_h, \p_d^2 u_h   & 1/2& 2/3   \\ \hline
 Du_h,D\p_d u_h, D\p_d^2 u_h   & 5/6 & 1/(1+r)   \\ \hline
\p_d^3 u_h   & 1   &1 \\ \hline
 u_d,\p_d u_d, \p_d^2 u_d, \p_d^3 u_d    & 5/6 & 1/(1+r)  \\ \hline
\end{array}
 \end{split}
\eeq

\item[(2)] The following tables encode the powers in the $L^2$ and $L^\infty$ interpolation estimates for the nonlinear terms $G^i$ and their derivatives in terms of $\senn$:
\beq\label{t3}
\begin{split}
\quad\begin{array}[t]{| c | c  | c |}
\hline
 L^2   \mid  \senn   &    2D & 3D \\ \hline
 G^{1}   & 3/4 & 1   \\ \hline
\nabla G^{1}  & 1 & 1\\ \hline
 G_d^{1}   & 1  & 1  \\ \hline
  G^{2}  &
1 & 1   \\ \hline
  G^{3}  & 1 & 1\\  \hline
 G^{4}  & 1 & 1\\  \hline
\end{array}
 \quad
 \begin{array}[t]{| c | c  | c |}
\hline
 L^\infty   \mid  \senn   &    2D & 3D  \\ \hline
 G^{1}   & 1  & 1   \\ \hline
  G^{2}  &
1  & 1   \\ \hline
  G^{3}  & 1& 1  \\  \hline
 G^{4}  & 1 & 1 \\  \hline
\end{array}
\end{split}
\eeq

  The following tables encode the powers in the $L^2$ and $L^\infty$ interpolation estimates for the nonlinear terms $G^i$ and their derivatives in terms of $\sdnn$:
\beq\label{t4}
\begin{split}
\begin{array}[t]{| c | c  | c |}
\hline
 L^2   \mid  \sdnn   &    2D & 3D \\ \hline
 G^{1}   & 1/2   & 2/3  \\ \hline
 \nabla G^{1}   & 5/6   & 1  \\ \hline
\nabla^2 G^{1}   & 1   & 1  \\ \hline
 G_d^{1}   & 5/6   & 1  \\ \hline
 \nabla G_d^{1}   & 1  & 1   \\ \hline
  G^{2}  &
5/6   & 1  \\ \hline
\nabla  G^{2}  &
1    & 1 \\ \hline
  G^{3}  & 5/6 & 1\\  \hline
  D G^{3}  & 1  & 1\\  \hline
 G^{4}  & 5/6 & 1 \\  \hline
 D G^{4}  & 1 & 1 \\  \hline
\end{array}
 \quad
\begin{array}[t]{| c | c  | c |}
\hline
 L^\infty   \mid  \sdnn   &    2D & 3D   \\ \hline
 G^{1}   & 2/3  & 1 \\ \hline
 \nabla G^{1}   & 1 & 1  \\ \hline
 G^{1}_d   & 1  & 1 \\ \hline
  G^{2}  &
1   & 1 \\ \hline
  G^{3}  & 1& 1\\  \hline
 G^{4}    & 1 & 1\\  \hline
\end{array}
\end{split}
\eeq
\end{itemize}
\end{lem}
\begin{proof}
The interpolation powers $\theta$ recorded in the above tables of \eqref{t1}--\eqref{t4} have been determined by using the full structure of the linear parts and the nonlinear terms $G^i$ in the equations \eqref{linear_perturbed}. We must record estimates for too many choices of $X$ to allow us to write the full details of each estimate. However, most of the estimates are straightforward, so we will present only a sketch of how to obtain them, providing details only for the most delicate estimates.

The procedure is basically as follows. First, the definition  \eqref{i_energy_min_2} of $\senn$ and the definition \eqref{p_ldissipation_def} of $\sdnn$, Sobolev embeddings and Lemmas \ref{p_poisson}--\ref{i_slice_interp} give some preliminary estimates of $u,p,\eta,\bar\eta$ (and some of their derivatives), which may have smaller powers $\theta$ than those recorded in the tables of \eqref{t1}--\eqref{t2}.  With these estimates of $u,p,\eta,\bar\eta$, we then estimate the  $G^i$ terms. The definitions \eqref{Gi_def_start}--\eqref{Gi_def_end} of $G^{i}$ show that these terms are linear combinations of products of one or more terms that can be estimated in either $L^2$ or $L^\infty$. For the $L^\infty$ tables  of \eqref{t3}--\eqref{t4} we estimate products with $ \ns{ XY}_{L^\infty} \le \ns{X}_{L^\infty} \pnorm{Y}{\infty}$; for the $L^2$ tables, we estimate products with both
$
 \ns{ XY}_{0} \le \ns{X}_{0} \pnorm{Y}{\infty} \text{ and } \ns{ XY}_{0} \le \ns{Y}_{0} \pnorm{X}{\infty}$,
and then take the larger value of $\theta$ produced by these two bounds, where Lemma \ref{i_N_constraint} will be used implicitly. These gives some preliminary estimates of $G^i$ (and some of their derivatives), which may have smaller powers $\theta$ than those recorded in the tables of \eqref{t3}--\eqref{t4}. With these estimates of $G^i$, we then use the linear structure of  \eqref{linear_perturbed} to improve the estimates of $u$, $\nabla p$, etc, which in turn improve the estimates of $G^i$. Such iterative scheme can be carried out repeatly until that we have the desired powers $\theta$ as recorded in these tables of \eqref{t1}--\eqref{t4}. The procedure is mostly straightforward, and below we explain only how to determine the powers $\theta$ in a bit more details.

Before proceeding further, we may explain the different powers in the  tables  \eqref{t1}--\eqref{t2}, in terms of $L^2$ or $L^\infty$, $\senn$ or $\sdnn$, and $2D$ or $3D$. First, in the  $L^2$ tables, basically a $1$ spatial derivative count contributes a $1/2$ power of $\senn $ and a $1/3$ power of $\sdnn $; this results mainly the differences between the first tables of \eqref{t1} and \eqref{t2}, and also the differences between the first tables of \eqref{t3} and \eqref{t4}.
Second, by Sobolev embeddings, an $L^\infty$ norm contributes an $L^2$ norm  of a $1/2$ horizontal derivative count in $2D$ and of a $1$ horizontal derivative count in $3D$ (or $1/2^-$ and $1^-$ due to the $L^\infty$ limiting embedding cases); this results mainly the differences between the two tables of \eqref{t1} and \eqref{t2}, respectively. Note that the different powers in the  tables  of \eqref{t1}--\eqref{t2} reflects the ones in the  tables in  \eqref{t3}--\eqref{t4}.

Now we start our estimates. First, the powers of $\eta,\ D\eta$, $D^2\eta$, $\bar\eta,\ D\bar\eta$, $D^2\bar\eta$ and $Dp$, $D^2 p$ as recorded in these tables of \eqref{t1}--\eqref{t2} follows by Sobolev embeddings and Lemmas  \ref{p_poisson}--\ref{i_slice_interp}, and they can not be improved.  

Next, we estimate $u_h$. We use the horizontal component of the first equation in \eqref{linear_perturbed} to find
\beq
\p_d^2 u_h\sim  Dp+D^2u_h +\dt u_h+G^{1,1}_h+G^{1,2}_h+D\bar\eta D\p_d u_h+G^{1,4}_h+G^{1,5}_h.
\eeq
It is then straightforward to check that the the powers of $\p_d ^2 u_h$ in these tables of \eqref{t1}--\eqref{t2} are determined by those of  $Dp$. On the other hand, we use the horizontal component of the third equation in \eqref{linear_perturbed} to find
\beq
\p_d u_h \sim  Du_d+ D\eta \nabla u \quad \text{on }\Sigma.
\eeq
Since $u_h=0$ on $\Sigma_b$, this together with Poincar\'e's inequality of Lemmas \ref{poincare_b}--\ref{poincare_trace} and the trace theory allows us to check that the the powers of $u_h, \p_d  u_h$ in these tables of \eqref{t1}--\eqref{t2} are determined by those of  $\p_d^2 u_h$ and hence $Dp$. Note that in particular, the proof of \eqref{t1} is completed. Similarly, the powers of $Du_h, D\p_d  u_h, D\p_d ^2 u_h$ in the two tables of \eqref{t2} are determined by those of  $D^2p$.  

We then estimate $u_d$. We use the second equation in \eqref{linear_perturbed} to find
\beq\p_d  u_d \sim  D u_h + D\bar\eta \p_d  u_h.
\eeq
Since $u_d=0$ on $\Sigma_b$, this together with Poincar\'e's inequality of Lemma \ref{poincare_trace}  implies that the powers of $u_d , \p_d u_d , \p_d ^2u_d , \p_d ^3u_d $ in the two tables of \eqref{t2} are determined by those of  $Du_h, D\p_d  u_h$, $D\p_d ^2 u_h$. 

We now estimate $\p_d  p$. We use the vertical component of the first equation in \eqref{linear_perturbed} to find
\beq
\p_d p \sim    \nabla^2u_d  + \dt u_d  +G^{1,2}_d +G^{1,3}_d  +G^{1,4}_d +G^{1,5}_d .
\eeq
We can check that the the powers of $\p_d p,\p_d ^2p$  in the two tables of \eqref{t2} are determined by those of $\p_d ^2u_d , \p_d ^3u_d $. 

Now we estimate $\p_d ^3 u_h$. We use the horizontal component of the first equation in \eqref{linear_perturbed} to find
\beq \p_d ^3 u_h=\p_d  D^2 u_d +\p_d  Dp+\p_d \dt u_d +G^{1,1}_d +G^{1,2}_d +G^{1,3}_d  +G^{1,4}_d +G^{1,5}_d. 
\eeq
We can check that the the powers of $\p_d ^3 u_h$  in the two tables of \eqref{t2} are determined by those of $G^{1,1}_d \sim \bar\eta \p_d p$, which are the sum of power of $\p_d p$ in the same table and power of $\bar\eta$  in the second table of \eqref{t2}. 

Finally, we estimate $\dt\eta,\dt\bar\eta$ and $\dt p$. We use the fourth equation in \eqref{linear_perturbed} to find
\beq\dt\eta \sim  u_d+ G^4 \quad \text{on }\Sigma.
\eeq
 Then the trace theory Lemma   \ref{p_poisson} imply that the the powers of $\dt\eta,\dt\bar\eta$ in the two tables of \eqref{t2}  are determined by those  of $u_d$. On the other hand, we use the vertical component of the third equation in \eqref{linear_perturbed} to find
\beq \dt p \sim  \dt \eta+2\dt \p_du_d+ \dt G^3_d \quad \text{on }\Sigma.
\eeq
This together with Poincar\'e's inequality of Lemma \ref{poincare_b} yields that the the powers of $\dt p$ in the two tables of \eqref{t2} are   determined by those of  $\dt\eta$. 
Note that the proof of   \eqref{t2} is thus completed.

With the estimates \eqref{t1}--\eqref{t2}  in hand, it is then fairly routine to prove these tables of \eqref{t3}--\eqref{t4}.
\end{proof}
\begin{lem}
Note that most of the interpolation powers $\theta$ in $3D$ of Lemma \ref{le inter}  improve those of Lemma 3.1--Proposition 3.16 in \cite{GT_inf}. This is due to that our low-order energy $\senn$ is stronger than $\se{N+2,2}$, which results also that our analysis is much more simple and direct than those in \cite{GT_inf}. 
\end{lem}

Now we record the interpolation estimates for $\k$, $\bar\k$ and $\mathcal{K}$, as defined by \eqref{kdef}, \eqref{bkdef} and \eqref{kkkdef}, respectively.
\begin{lem}\label{klem}
It holds that
\beq\label{kes}
\mathcal{K}\ls\k \ls \fe{2N}^{\kappa_d/2} \senn^{1-\kappa_d/2},\ d=2,3
\eeq
and 
\beq\label{bkes}
\bar\k \ls  \senn .
\eeq
\end{lem}
\begin{proof}
\eqref{kes} follows directly from the definition  \eqref{i_energy_min_2} of $\senn$ and the $L^\infty$ table  in \eqref{t1} with the choice $r=\kappa_3/(2-\kappa_3)$ when $d=3$. On the other hand, we use the horizontal component of the third equation in \eqref{linear_perturbed} to find
\beq
\p_d u_h =-  Du_d- G^3_h \text{ on }\Sigma.
\eeq
\eqref{bkes} then follows  from the definition  \eqref{i_energy_min_2} of $\senn$ and the $L^2$ table  in \eqref{t3}. 
\end{proof}
\subsection{Nonlinear estimates}

 We now present the estimates of the nonlinear terms $G^i$. We first record the estimates at the $N+2$ level.
 \begin{lem}\label{p_G_estimates_N+2}
Let $G^i$ be defined by \eqref{Gi_def_start}--\eqref{Gi_def_end}. Then
\begin{enumerate}
\item It holds that
\begin{align}\label{p_G_e_h_0}
&\ns{ \bar{\nab}_1^{2(N+2)-2} G^1}_{0} +  \ns{ \bar{\nab}_0^{2(N+2)-2}  G^2}_{1} +
 \as{ \bar{D}_{0}^{2(N+2)-2} G^3}_{1/2}\nonumber
\\&\quad+ \as{\bar{D}_{0}^{2(N+2)-2} G^4}_{1/2}\ls \fe{2N}\senn.
\end{align}
\item It holds that
\begin{align}\label{p_G_d_h_0}
&\ns{ \bar{\nab}_2^{2(N+2)-1} G^1}_{0} +  \ns{ \bar{\nab}_1^{2(N+2)-1}  G^2}_{1}  +
 \as{ \bar{D}_{1}^{2(N+2)-1} G^3}_{1/2}\nonumber
\\&\quad + \as{\bar{D}_{1}^{2(N+2)-1 } G^4}_{1/2}+ \as{\bar{D}_{0}^{2(N+2)-2} \dt G^4}_{1/2}
\ls
\fe{2N}\sdnn  .
\end{align}
\end{enumerate}
\end{lem}
\begin{proof}
The estimates of these nonlinearities are fairly routine to derive: we note that all $G^i$ terms are quadratic or of higher order; then we apply the differential operator and expand using the Leibniz rule; each term in the resulting sum is also at least quadratic, and we estimate one term in $H^k$ ($k=0$ or $1/2$ depending on $G^i$) and the other term in $L^\infty$ or $H^{m}$ for $m$ depending on $k$, using Sobolev embeddings, trace theory, and Lemmas  \ref{p_poisson}--\ref{i_slice_interp} and \ref{i_sobolev_product_1}.

Note that the derivative count in the differential operators is chosen in order to allow estimation by $\senn $ in \eqref{p_G_e_h_0} and by $\sdnn $ in \eqref{p_G_d_h_0}. Because $\senn $ and $\sdnn $ involve minimal derivative counts, there may be terms in the sum $\pa G^i$ that cannot be directly estimated, and we must appeal to the interpolation results in the $L^2$ tables of \eqref{t3} and \eqref{t4} in Lemma \ref{le inter}. This yields directly the estimate \eqref{p_G_e_h_0} by using the $L^2$ table of \eqref{t3} as the derivative count involved also does not exceed those in $\senn $. For the estimate \eqref{p_G_d_h_0}, by using the $L^2$ table of \eqref{t4} it suffices to estimate three exceptions whose derivative count involved  exceed those in $\sdnn $:  $\ns{\nabla^{2(N+2)+1} \bar\eta \nabla u}_{0}$ when estimating $G^1$ and $G^2$, $\as{D^{2(N+2)-1}\nabla\bar\eta \nabla u}_{1/2}$ when estimating $G^3$ and $\as{D^{2(N+2)} \eta  u}_{1/2}$ when estimating $G^4$.
To control them, we use the Sobolev interpolation to have
\begin{align}\label{ii1}
 \as{D^{2(N+2)} \eta}_{1/2}\nonumber&  \le \(\as{D^{3} \eta}_{2(N+2)-7/2}\)^{(4N-11)/(4N-9) } \(\as{D^{3} \eta}_{4N-4}\)^{2/(4N-9)}
\\& \le \sdnn^{(4N-11)/(4N-9) } \fe{2N}^{2/(4N-9)}.
\end{align}
On the other hand, the $L^\infty$ table in \eqref{t2} implies in particular that
\begin{equation}\label{ii2}
 \ns{u}_{C^2(\bar\Omega)} \ls \sdnn^{1/2}\fe{2N}^{1/2}.
\end{equation}
Hence, by \eqref{ii1}--\eqref{ii2} and Lemma \ref{i_sobolev_product_2}, we obtain that for $N\ge 4$,
\begin{align}
\as{D^{2(N+2)} \eta  u}_{1/2}&\ls \as{D^{2(N+2)}  \eta  }_{1/2}\as{ u}_{C^1(\Sigma)} \nonumber
\\& \ls \sdnn^{(4N-11)/(4N-9) } \fe{2N}^{2/(4N-9)}\sdnn^{1/2}\fe{2N}^{1/2}\ls
\fe{2N}\sdnn.
\end{align}
Similarly, by using additionally the trace theory and Lemma \ref{p_poisson}, we have
\begin{align}
\as{D^{2(N+2)-1}\nabla\bar\eta \nabla u}_{1/2}&\ls \as{D^{2(N+2)-1} \nabla \bar\eta}_{1/2} \as{ \nabla u}_{C^1(\Sigma)} \ls
 \ns{D^{2(N+2)-1} \nabla \bar\eta}_{1 }    \as{ \nabla u}_{C^1(\Sigma)}\nonumber
\\&  \ls
\as{D^{2(N+2)}  \eta  }_{1/2} \as{ \nabla u}_{C^1(\Sigma)}\ls
\fe{2N}\sdnn
\end{align}
and
\begin{align}
\ns{\nabla^{2(N+2)+1} \bar\eta \nabla u}_{0}&\ls \ns{\nabla^{2(N+2)+1} \bar\eta}_0\ns{ \nabla u}_{C(\bar\Omega)}\nonumber
\\&\ls
\as{D^{2(N+2)}  \eta  }_{1/2}\ns{ \nabla u}_{C(\bar\Omega)}\ls
\fe{2N}\sdnn.
\end{align}
We then conclude  \eqref{p_G_d_h_0}.
\end{proof}

Now we record the estimates at the $2N$ level.
\begin{lem}\label{p_G_estimates2N}
Let $G^i$ be defined by \eqref{Gi_def_start}--\eqref{Gi_def_end}. Then
\begin{enumerate}
\item It holds that
\begin{align}\label{p_G_e_0}
& \ns{  {\nab}_0^{4N-3} G^1}_{0}   +  \ns{  {\nab}_0^{4N-3}  G^2}_{1} +
 \as{  {D}_{0}^{4N-3}  G^3}_{1/2}+
 \as{ {D}_{0}^{4N-3} G^4}_{1/2}
 \\&\quad+\ns{ \bar{\nab}_0^{4N-4} \dt G^1}_{0}   +  \ns{ \bar{\nab}_0^{4N-4} \dt   G^2}_{1} +
 \as{ \bar{D}_{0}^{4N-4} \dt G^3}_{1/2}+
 \as{\bar{D}_{0}^{4N-4} \dt  G^4}_{1/2} \ls(\fe{2N})^2,\nonumber
 \end{align}
and
\begin{align}\label{p_G_e_1}
& \ns{ \nabla^{4N-2} G^1}_{0}   +  \ns{\nabla^{4N-2}G^2}_{1} +
 \as{ \nabla^{4N-2}  G^3}_{1/2}+
 \as{\nabla^{4N-2}G^4}_{1/2}\nonumber
\\&\quad \ls(\fe{2N})^2+\k\sem{2N}.
 \end{align}
\item It holds that
\begin{align}\label{p_G_d_0}
& \ns{  {\nab}_0^{4N-3} G^1}_{0}   +  \ns{  {\nab}_0^{4N-3}  G^2}_{1} +
 \as{  {D}_{0}^{4N-3}  G^3}_{1/2}+
 \as{ {D}_{0}^{4N-3} G^4}_{1/2}
 \\&\quad+\ns{ \bar{\nab}_0^{4N-3} \dt G^1}_{0}   +  \ns{ \bar{\nab}_0^{4N-3} \dt   G^2}_{1} +
 \as{ \bar{D}_{0}^{4N-3} \dt G^3}_{1/2}+
 \as{\bar{D}_{0}^{4N-2} \dt  G^4}_{1/2} \ls\fe{2N} \fd{2N},\nonumber
 \end{align} 
\begin{align}\label{p_G_d_1} 
& \ns{ \nabla^{4N-2} G^1}_{0}   +  \ns{\nabla^{4N-2}G^2}_{1} +
 \as{ D^{4N-2}  G^3}_{1/2}+
 \as{D^{4N-2}G^4}_{1/2}\nonumber
\\&\quad \ls \fe{2N} \fd{2N}+\k\sdm{2N},
 \end{align}
and
\begin{align}\label{p_G_d_2} 
& \ns{ \nabla^{4N-1} G^1}_{0}   +  \ns{\nabla^{4N-1}G^2}_{1} +
 \as{ D^{4N-1}  G^3}_{1/2}+
 \as{D^{4N-1}G^4}_{1/2}\nonumber
\\&\quad \ls \fe{2N}  \sdm{2N} +\k \f,
 \end{align}
\end{enumerate}
\end{lem}
\begin{proof}
The proof is similar to Lemma \ref{p_G_estimates_N+2}. The estimate  \eqref{p_G_e_0} is straightforward since $\fe{2N}$  has no minimal derivative restrictions, and there is no problem for the estimate  \eqref{p_G_d_0} even $\eta$ and $\bar\eta$ appearing in $G^i$ are not included in $\fd{2N}$, since we can always control them by $\fe{2N}$ and other terms multiplying them by $\fd{2N}$.

We now turn to the derivation of the estimate  \eqref{p_G_d_2}.  With three exceptions, we may argue as in the derivation of \eqref{p_G_d_0} to estimate the desired norms of all terms  by $\fe{2N} \sdm{2N}$.  The exceptional terms are $\ns{\nabla^{4N+1} \bar\eta \nabla u}_{0}$ when estimating $G^1$ and $G^2$, $\as{D^{4N-1} \nabla \bar\eta \nabla u}_{1/2}$  when estimating $G^3$ and $\as{D^{4N}  \eta  u}_{1/2}$ when estimating $G^4$.  We will now show how to estimate the exceptional terms with $\k \f$. Indeed, by Lemma \ref{i_sobolev_product_2}, we have
\begin{align}
 \as{D^{4N}  \eta  u}_{1/2} \ls  \as{D^{4N}  \eta  }_{1/2}\as{ u}_{C^1(\Sigma)} 
 \ls \f \k.
\end{align}
Similarly, by using additionally the trace theory and Lemma \ref{p_poisson}, we have
\begin{align}
\as{D^{4N-1} \nabla \bar\eta \nabla u}_{1/2} &\ls \as{D^{4N-1} \nabla \bar\eta}_{1/2} \as{ \nabla u}_{C^1(\Sigma)}   
  \ls
\ns{ \bar\eta }_{4N+1} \k\ls
\as{ \eta }_{4N+1/2} \k\ls \f \k 
\end{align}
and
\begin{equation}\label{iGe_3}
 \ns{\nabla^{4N+1} \bar\eta \nabla u}_{0}\ls \ns{\nab^{4N+1} \bar{\eta}}_{0} \ns{\nab u}_{C(\bar\Omega)}  \ls
\as{ \eta }_{4N+1/2} \k\ls \f \k.
\end{equation}
We then conclude \eqref{p_G_d_2}.

Finally,  the estimates \eqref{p_G_e_1} and \eqref{p_G_d_1} follow similarly as \eqref{p_G_d_2} by bounding $\as{\eta}_{4N-1/2}$ by $\sem{2N}$ in \eqref{p_G_e_1} and by $\sdm{2N}$ in \eqref{p_G_d_1}.
\end{proof}

\section{Energy evolution}\label{sec global}

\subsection{Energy evolution in perturbed linear form}

To derive the energy evolution of the mixed time-horizontal spatial derivatives of the solution to \eqref{geometric}, that is, excluding the highest order time derivative and the highest order horizontal spatial derivatives, we shall use the perturbed linear formulation \eqref{linear_perturbed}. 
Recall that we employ the derivative conventions  \eqref{deco1}--\eqref{deco3} from Section \ref{nota}.

We first record the estimates for the evolution of the energy at the $2N$ level.
\begin{prop}\label{p_upper_evolution 2N}
It holds that
\begin{align}\label{p_u_e_00}
& \frac{d}{dt} \(\ns{ \bar{D}_{0}^{4N-1}  u}_{0} + \ns{    \bar{D}^{4N-3} D\dt  u}_{0} + \as{ \bar{D}_{0}^{4N-1}  \eta}_{0} + \as{    \bar{D}^{4N-3} D\dt  \eta}_{0}\)  \nonumber\\&\quad+ \ns{ \bar{D}_{0}^{4N-1}  u}_{1} + \ns{    \bar{D}^{4N-3} D\dt  u}_{1}  \ls    \sqrt{\fe{2N}}  \fd{2N}  +\sqrt{ \fd{2N} \k\sdm{2N} } .
\end{align}
\end{prop}
\begin{proof}
Let $\al\in \mathbb{N}^{1+(d-1)}$ be so that $1\le |\al|\le 4N$ and $1\le |\al_h|\le 4N-1$. Applying $\partial^\alpha$  to the first equation of \eqref{linear_perturbed} and then taking the dot product with $\pa u$, using the other equations of \eqref{linear_perturbed} as in Proposition 6.2 of \cite{GT_inf}, we find that
\begin{align} \label{p_u_e_111}
 &\hal \frac{d}{dt}\left(  \int_\Omega \abs{\pa{u}}^2  +    \int_\Sigma \abs{\pa \eta}^2  \right)
+ \hal \int_\Omega \abs{\sg \pa{u}}^2\nonumber
 \\&\quad= \int_\Omega    \pa u  \cdot (\pa G^1-\nabla \pa G^2) +\int_\Omega    \pa p \pa G^2
  + \int_\Sigma -\pa u \cdot \pa G^3 +  \pa \eta  \pa G^4.
\end{align}

We now estimate the right hand side of \eqref{p_u_e_111}.  Since $1\le |\al_h|\le 4N-1$, we may write $\alpha = \beta +(\alpha-\beta)$ for some $\beta \in \mathbb{N}^{d-1}$ with $\abs{\beta}=1$.  Hence $\abs{\alpha-\beta} \le 4N-1$ and $|\alpha_h-\beta_h|\le 4N-2$, we can then integrate by parts and use the estimates \eqref{p_G_d_0}--\eqref{p_G_d_1} of Lemma \ref{p_G_estimates2N} to have
\begin{align}\label{i_de_4}
\abs{ \int_\Omega      \pa  u \cdot  (\pa G^1-\nabla \pa G^2) } &= \abs{ \int_\Omega      \p^{\alpha+\beta}  u \cdot    (\p^{\alpha-\beta} G^1-\nabla \p^{\alpha-\beta} G^2) }
 \nonumber\\&\le \norm{\p^{\alpha}  u}_{1} \( \norm{ \p^{\alpha-\beta} G^1 }_{0}+\norm{ \p^{\alpha-\beta} G^2 }_{1} \)\nonumber\\&
\ls \sqrt{\fd{2N} } \sqrt{ \fe{2N} \fd{2N}  + \k\sdm{2N}} .
\end{align}
For the $G^2$ term we do not need to integrate by parts: by \eqref{p_G_d_0}--\eqref{p_G_d_1},
\begin{align}\label{i_de_5}
\abs{ \int_\Omega    \pa p \pa G^2 }
&\le \norm{\pa p}_{0}  \norm{ \p^{\alpha-\beta}\p^\beta G^2 }_{0}
\le \norm{\pa p}_{0}  \norm{ \p^{\alpha-\beta} G^2 }_{1}\nonumber
  \\&\ls \sqrt{\fd{2N} } \sqrt{ \fe{2N} \fd{2N}  + \k\sdm{2N} } .
\end{align}
For the $G^3$ term we integrate by parts and use the trace estimate to see that, by \eqref{p_G_d_0}--\eqref{p_G_d_1},
\begin{align}\label{i_de_6}
\abs{ \int_\Sigma     \pa  u \cdot   \pa G^3 } &= \abs{ \int_\Sigma  \p^{\alpha+\beta}  u \cdot   \p^{\alpha-\beta} G^3 }
 \le \abs{\p^{\alpha+\beta}  u}_{-1/2}   \abs{ \p^{\alpha-\beta} G^3 }_{1/2}\nonumber \\
&\le \abs{\p^{\alpha }  u}_{1/2}   \abs{ \p^{\alpha-\beta} G^3 }_{1/2}
 \le \norm{\p^{\alpha}  u}_{1}  \abs{ \p^{\alpha-\beta} G^3 }_{1/2} \nonumber
  \\&\ls \sqrt{\fd{2N} } \sqrt{ \fe{2N} \fd{2N}  + \k\sdm{2N} }.
\end{align}
For the $G^4$ term we must split to two cases: $\alpha_0\ge 1$ and $\alpha_0 =0$.  In the former case, there is at least one temporal derivative in $\pa$, so by \eqref{p_G_d_0} we have
\begin{equation}\label{i_de_71}
\abs{ \int_\Sigma     \pa \eta    \pa G^4 }     \le \abs{\bar{D}^{4N-2}_0 \dt \eta }_{0}\abs{\bar{D}^{4N-2}_0 \dt G^4 }_{0}
 \ls \sqrt{\fd{2N} } \sqrt{ \fe{2N} \fd{2N}  }.
\end{equation}
In the latter case, there involves only spatial derivatives, that is, $1\le \abs{\al}=|\al_h|\le 4N-1$, and we write
\begin{equation}
 \pa G^4 = \pa ( D  \eta \cdot u_h) =  -  D  \pa \eta \cdot   u_h-\[\pa, u_h\] \cdot D\eta.
 \end{equation}
We use the integration by parts to see that, by Lemma \ref{i_sobolev_product_2},
\begin{align}\label{e23}
 &-\int_\Sigma  \pa \eta    D  \pa \eta \cdot u_h
 = -\hal\int_\Sigma      D  \abs{\pa \eta}^2 \cdot u_h= \hal\int_\Sigma  \abs{\pa  \eta}^2 D \cdot u_h
\nonumber  \\& \quad\ls  \abs{\pa  \eta}_{-1/2} \abs{\pa  \eta D \cdot u_h}_{1/2} \ls \abs{\pa  \eta}_{-1/2} \abs{\pa  \eta}_{1/2} \abs{ D   u_h}_{C^1(\Sigma)}\nonumber  \\& \quad\ls     \sqrt{\fd{2N}   \sdm{2N} \k}.
\end{align}
On the other hand, similarly as the derivation of \eqref{p_G_d_1} in Lemma \ref{p_G_estimates2N}, we have
\begin{equation}
 \as{ \[\pa, u\] D\eta}_{1/2}
 \ls  { \fe{2N} \fd{2N}  + \k\sdm{2N}} 
\end{equation}
and hence
\begin{align}\label{e2312}
 \int_\Sigma    \pa \eta    \[\pa, u\] D\eta
 &\ls \abs{\pa \eta }_{-1/2}   \abs{ \[\pa, u\] D\eta}_{1/2}\nonumber
 \\&\ls \sqrt{\fd{2N} }\sqrt{ \fe{2N} \fd{2N}  + \k\sdm{2N} }.
 \end{align}
In light of \eqref{e23} and \eqref{e2312}, we conclude that
\begin{align}\label{i_de_7}
\abs{ \int_\Sigma     \pa  \eta \cdot   \pa G^4 }  \ls \sqrt{\fd{2N} } \sqrt{ \fe{2N} \fd{2N}  + \k\sdm{2N} }.
\end{align}

Now, by the estimates \eqref{i_de_4}--\eqref{i_de_71}  and \eqref{i_de_7}, we deduce from \eqref{p_u_e_111} that for all $\al\in \mathbb{N}^{1+(d-1)}$ with $1\le |\al|\le 4N$ and $1\le |\al_h|\le 4N-1$,
\begin{align}  \label{p_u_e_9}
 &\hal \frac{d}{dt}\left(  \int_\Omega \abs{\pa{u}}^2  +    \int_\Sigma \abs{\pa \eta}^2  \right)
+ \hal \int_\Omega \abs{\sg \pa{u}}^2\nonumber
 \\&\quad\ls  \sqrt{\fd{2N} } \sqrt{ \fe{2N} \fd{2N}  + \k\sdm{2N} }
  \ls  \sqrt{\fe{2N}} \fd{2N}  +\sqrt{ \fd{2N} \k\sdm{2N} } .
\end{align}
The estimate \eqref{p_u_e_00} then follows from \eqref{p_u_e_9} by summing over such $\alpha$ and \eqref{intro0}, using Korn's inequality of Lemma \ref{i_korn} since $\pa u=0$ for $\al \in \mathbb{N}^{1+(d-1)}$.
\end{proof}

We then record the estimates for the evolution of the energy at the $N+2$ level. 

\begin{prop}\label{p_upper_evolution N+2}
It holds that
\begin{align}\label{p_u_e_11}
&\frac{d}{dt} \(\ns{ \bar{D}_{2}^{2(N+2)-1}  u}_{0} + \ns{    \bar{D}^{2(N+2)-1} D  u}_{0} +\as{ \bar{D}_{2}^{2(N+2)-1}  \eta}_{0} + \as{    \bar{D}^{2(N+2)-1} D  \eta}_{0}\)  \nonumber\\&\quad+ \ns{ \bar{D}_{2}^{2(N+2)-1}  u}_{1} + \ns{    \bar{D}^{2(N+2)-1} D  u}_{1}    \ls   \sqrt{\fe{2N}} \sdnn     .
\end{align}
\end{prop}
\begin{proof}
This is just a restatement of Proposition 6.4 of \cite{GT_inf} when $m=2$.  Let $\al\in \mathbb{N}^{1+(d-1)}$ be so that $2\le |\al|\le 2(N+2)$ and $  \al_0 \le N+1$.  Note that \eqref{p_u_e_111} holds, and the right hand side of \eqref{p_u_e_111} can be estimated similarly as in Proposition \ref{p_upper_evolution 2N} by using \eqref{p_G_d_h_0} in place of \eqref{p_G_d_0}--\eqref{p_G_d_1} so that they can be bounded by $ \sqrt{\fe{2N}} \sdnn$, except the following terms:
\begin{align}\label{eexx1}
  \int_\Sigma   \pa \eta  \pa G^4 \text{ when }|\al|=2(N+2)\text{ and }\al_0=0
\end{align}
and 
\begin{align}\label{eexx2}
  \int_\Sigma   \pa \eta  \pa G^4 \text{ and }\int_\Omega   \pa p  \pa G^2 \text{ when }|\al|=2.
\end{align}

If $|\al|=2(N+2)$ and $\al_0=0$, then we have
\begin{align} \label{tta1}
& \int_\Sigma   \pa \eta  \pa G^4  \ls \abs{D^{2(N+2)}\eta}_0\abs{D^{2(N+2)}G^4}_0\nonumber
\\& \quad\ls \abs{D^{2(N+2)}\eta}_0\(\abs{D^{2(N+2)}D\eta}_0\abs{u_h}_{L^\infty(\Sigma)}
 +\abs{ D\eta}_{L^\infty(\Sigma)}  \abs{D^{2(N+2)} u_h}_{0}\).
\end{align}
To control the right hand side of the above, we use the Sobolev interpolation to have
\begin{align}\label{tta2}
 \as{D^{2(N+2)} \eta}_{0}\nonumber&  \le \(\as{D^{3} \eta}_{2(N+2)-7/2}\)^{(4N-10)/(4N-9) } \(\as{D^{3} \eta}_{4N-4}\)^{1/(4N-9)}
\\& \le \sdnn^{(4N-10)/(4N-9) } \fe{2N}^{1/(4N-9)} 
\end{align}
and 
\begin{align}\label{tta3}
 \as{D^{2(N+2)+1} \eta}_{0}\nonumber&  \le \(\as{D^{3} \eta}_{2(N+2)-7/2}\)^{(4N-10)/(4N-9) } \(\as{D^{3} \eta}_{4N-4}\)^{1/(4N-9)}
\\& \le \sdnn^{(4N-12)/(4N-9) } \fe{2N}^{3/(4N-9)}.
\end{align}
On the other hand, the $L^\infty$ table in \eqref{t2} implies in particular that
\begin{equation}\label{tta4}
\as{u_h}_{C(\Sigma)}\ls \sdnn^{1/2}\fe{2N}^{1/2}
\text{ and }\as{ D\eta}_{C(\Sigma)}  \ls \sdnn^{1/2}\fe{2N}^{1/2}.
\end{equation}
Hence, by \eqref{tta2}--\eqref{tta4}, we obtain from \eqref{tta1} that for $N\ge 5$,
\begin{align} \label{tta5}
 \int_\Sigma   \pa \eta  \pa G^4& \ls\sqrt{\sdnn^{(4N-10)/(4N-9) } \fe{2N}^{1/(4N-9)}} \sqrt{ \sdnn^{1/2}\fe{2N}^{1/2} }\sqrt{ \sdnn^{(4N-12)/(4N-9) } \fe{2N}^{3/(4N-9)} 
 + \sdnn}\nonumber
 \\&\ls \sqrt{\fe{2N} }\sdnn .
\end{align}

Now for $|\al|=2$, if $\al_0=0$, then by integrating by parts and the $L^2$ table in \eqref{t4}, we have
\begin{align} \label{tta11}
&\int_\Sigma   \pa \eta  \pa G^4+ \int_\Omega   \pa p  \pa G^2  \ls\abs{D^{3}\eta}_0\abs{D G^4}_0+ \norm{D^{3}p}_0\norm{D G^2}_0\nonumber
\ \\&\quad  \ls\sqrt{   \sdnn}\sqrt{\fe{2N} \sdnn}\ls \sqrt{\fe{2N} }\sdnn.
\end{align}
If  $\al_0=1$, then by the two  tables in \eqref{t2}, we have
\begin{align} \label{tta12}
\int_\Sigma  \dt\eta \dt G^4& \ls  \abs{\dt \eta}_0\abs{\dt G^4}_0\ls \abs{\dt \eta}_0
\(\abs{\dt D\eta}_0\abs{u_h}_{C(\Sigma)}
 +\abs{ D\eta}_{C(\Sigma)}  \abs{\dt u_h}_{0}\)\nonumber
\ \\&  \ls\sqrt{\sdnn^{2/3} \fe{2N}^{1/3}}  \sqrt{   \sdnn}\sqrt{ \sdnn^{1/2}\fe{2N}^{1/2} }\ls \sqrt{\fe{2N} }\sdnn.
\end{align}
For the $G^2$ term, we need to use the structure of $G^2$ when $d=2$; indeed, we can treat this term unifiedly for $d=2,3$. Recall the Piola identity $\p_j (J\a_{ij})=0$. Then by the second equation in \eqref{geometric}, we have
$
\p_j(J\a_{ij}u_i)=0,
$
which implies that
\beq
G^2=\diverge u=\p_j((\delta_{ij}-J\a_{ij})u_i)=- D\!\cdot\!((J-1)u_h)+\p_d(D\varphi\cdott u_h).
\eeq
Hence, by integrating by parts in horizontal variable and the two tables in \eqref{t2}, we obtain
\begin{align} \label{tta13}
& \int_\Omega   \dt p  \dt G^2  =\int_\Omega  D \dt p   \!\cdot\!\dt((J-1)u_h)+  \int_\Omega   \dt p  \dt \p_d(D\varphi\cdott u_h)
\nonumber
\ \\&\quad\ls\norm{ D \dt p}_0\(\norm{\dt J}_0\norm{u_h}_{C(\bar\Omega)}+\norm{J-1}_{C(\bar\Omega)}\norm{\dt u_h}_0\)
\nonumber
\ \\&\qquad+\norm{   \dt p}_0\(\norm{   \dt  \p_d D\varphi}_0\norm{ u_h}_{C(\bar\Omega)}+ 
\norm{ \p_d D\varphi}_{C(\bar\Omega)}\norm{  \dt u_h}_0   \)
\nonumber
\ \\&\qquad+\norm{   \dt p}_0\( \norm{ \dt D\varphi}_0\norm{ \p_d u_h}_{C(\bar\Omega)}+ \norm{D\varphi}_{C(\bar\Omega)}\norm{ \dt \p_d u_h}_0  \)
\nonumber
\ \\&\quad  \ls\sqrt{   \sdnn}\(\sqrt{\sdnn^{5/6} \fe{2N}^{1/6}} \sqrt{ \sdnn^{1/2}\fe{2N}^{1/2} }+\sqrt{\fe{2N} }\sqrt{\sdnn}\)
\nonumber
\ \\&\qquad+
\sqrt{\sdnn^{2/3} \fe{2N}^{1/3}}  \(\sqrt{   \sdnn}\sqrt{ \sdnn^{1/2}\fe{2N}^{1/2} }+ \sqrt{\sdnn^{5/6} \fe{2N}^{1/6}} \sqrt{   \sdnn}\)
\nonumber
\ \\&\qquad+
\sqrt{\sdnn^{2/3} \fe{2N}^{1/3}}  \(\sqrt{   \sdnn}\sqrt{ \sdnn^{1/2}\fe{2N}^{1/2} }+\sqrt{ \sdnn^{1/2}\fe{2N}^{1/2} }\sqrt{   \sdnn}\)
\nonumber
\ \\&\quad\ls \sqrt{\fe{2N} }\sdnn.
\end{align}

Consequently, by \eqref{tta5}, \eqref{tta11}, \eqref{tta12} and \eqref{tta13}, we know that the exceptional terms \eqref{eexx1} and \eqref{eexx2} are all bounded by $\sqrt{\fe{2N} }\sdnn$, and hence that the right hand side of \eqref{p_u_e_111}  are bounded by $\sqrt{\fe{2N} }\sdnn$
for all $\al\in \mathbb{N}^{1+(d-1)}$ with $2\le |\al|\le 2(N+2)$ and $  \al_0 \le N+1$.  The estimate \eqref{p_u_e_11}
thus follows by summing over such $\al$.
\end{proof}

\subsection{Energy evolution using geometric form}

To derive the energy evolution of the highest order temporal derivatives of the solution to \eqref{geometric},  we will not use the perturbed linear formulation \eqref{linear_perturbed}.
As well explained in \cite{GT_inf}, if we did this  we would be unable to control the $G^1$, $G^2$ and $G^3$ terms in the right hand side of  \eqref{p_u_e_111}. Motivated by \cite{GT_inf},
we shall use the following geometric formulation. Applying the differential operator $\pa=\dt^{\alpha_0}$  to \eqref{geometric}, we find that
   \beq\label{linear_geometrica1}
   \begin{cases}    \D_{t} (\pa u) + u \cdot \nabla_\a  (\pa u) + \diva S_\a ( \pa p,
  \pa u)=F^1& \text{in }\Omega
      \\
      \nabla_\a  \cdot (\pa u) =F^2& \text{in }\Omega
      \\   S_\a ( \pa p,
\pa u)\n =     \p^\alpha \eta  \n+F^3 & \text{on }\Sigma
  \\   \partial_{t} (\p^\alpha \eta)=\pa u  \cdot \n +F^4& \text{on }\Sigma
  \\ \pa u =0 & \text{on }\Sigma_b,
 \end{cases}
       \eeq
       where
\begin{align}\label{F_def_start} 
 F^{1,\al}_i = &\sum_{0 < \beta \le \alpha } C_{\alpha}^{\beta}  \left\{ \p^\beta (  K\dt \varphi)      \p^{\alpha - \beta} \p_d u_i+  \mathcal{A}_{jk} \p_k (\p^\beta \mathcal{A}_{i\ell} \p^{\alpha - \beta} \p_\ell u_j  + \p^\beta \mathcal{A}_{j\ell} \p^{\alpha -\beta} \p_\ell u_i)  \right.
 \\  &\qquad\quad \left.+  \p^\beta \mathcal{A}_{j\ell} \p^{\alpha - \beta} \p_\ell (\mathcal{A}_{im} \p_m u_j + \mathcal{A}_{jm}\p_m u_i)  -   \p^\beta ( u_j  \mathcal{A}_{jk} ) \p^{\alpha - \beta} \p_k u_i  - \p^\beta \mathcal{A}_{ik} \p^{\alpha-\beta} \p_k p \right\},\nonumber
\end{align} 
\begin{equation}\label{i_F2_def}
 F^{2,\al} = - \sum_{0 < \beta \le \alpha} C_{\alpha}^{\beta}  \p^\beta \mathcal{A}_{ij} \p^{\alpha-\beta} \p_j u_i ,
\end{equation} 
\begin{equation} 
 F^{3,\al} = \sum_{0 < \beta \le \alpha} C_{\alpha}^{\beta}  \left\{\p^\beta D \eta \p^{\alpha-\beta}  (\eta - p )+
(  \p^\beta ( \n_j \mathcal{A}_{im} ) \p^{\alpha-\beta} \p_m u_j + \p^\beta ( \n_j \mathcal{A}_{jm} ) \p^{\alpha-\beta} \p_m u_i )\right\}
\end{equation}
and
\begin{equation}\label{F_def_end}
 F^{4,\al} =  \sum_{0 < \beta \le \alpha} C_{\alpha}^{\beta}  \p^\beta D \eta \cdott \p^{\alpha-\beta} u_h.
\end{equation}

We now present the estimates of these nonlinear terms $F^{i}$  at both $2N$ and $N+2$ levels.

\begin{lem}\label{p_F_estimates}
Let $F^{i,\al}$ be defined by \eqref{F_def_start}--\eqref{F_def_end}, then the following estimates hold.
\begin{enumerate}
\item 
Let $F^{i,2N}$ be $F^{i,\al}$ when $\dt^\al=\dt^{2N}$,  it holds that 
\begin{equation}\label{p_F_e_01}
 \ns{F^{1,2N} }_{0}+ \ns{\dt (J F^{2,2N} ) }_{0} + \as{F^{3,2N}}_{0} + \as{F^{4,2N}}_{0} \ls \fe{2N}  \fd{2N}
\end{equation}
and
\begin{equation}\label{p_F_e_02}
 \ns{F^{2,2N}}_{0} \ls  (\fe{2N})^2.
\end{equation}
\item Let $F^{i, N+2}$ be $F^{i,\al}$ when $\dt^\al=\dt^{N+2}$,  it holds that 
\begin{equation}\label{p_F_e_h_01}
 \ns{F^{1, N+2} }_{0}  + \ns{\dt (J F^{2, N+2} ) }_{0} + \as{F^{3, N+2}}_{0} + \as{F^{4, N+2}}_{0} \ls \fe{2N}   \sdnn 
\end{equation}
and
\begin{equation}\label{p_F_e_h_02}
 \ns{F^{2, N+2}}_{0}\ls \fe{2N}    \senn   .
\end{equation}
\end{enumerate}
\end{lem}
\begin{proof}
All these estimates, with the trivial replacement of $\se{2N}$ by $\fe{2N}$, etc., are recorded
in Theorems 5.1--5.2 of \cite{GT_inf}. Indeed, the proof of \eqref{p_F_e_h_01}--\eqref{p_F_e_h_02} will be simpler than those of Theorem 5.2 of \cite{GT_inf} as our low-order energy and dissipation are stronger than those in \cite{GT_inf}. For example,  
\begin{align}\label{i_F2_defes}
 \ns{F^{2,N+2}}_{0}& \ls  \sum_{0<\ell\le N+1}\ns{ \p_t^\ell \mathcal{A}_{ij}}_0\ns{ \p_t^{N+2-\ell} \p_j u_i}_{C(\bar\Omega)} +\ns{ \p_t^{N+2} \mathcal{A}_{ij}}_0\ns{   \p_j u_i}_{C(\bar\Omega)} \nonumber
\\& \ls  \fe{2N}    \senn +\senn \fe{2N}  \ls  \fe{2N}    \senn.
\end{align} 
This proves \eqref{p_F_e_h_02}.
\end{proof}

  We now present the evolution estimate for $2N$ temporal derivatives.
\begin{prop}\label{i_temporal_evolution 2N}
It holds that
\begin{equation} \label{tem en 2N}
   \frac{d}{dt} \left( \ns{ \dt^{2N} u}_{0}  +  \as{\dt^{2N}\eta}_{0}+2\int_\Omega J   \dt^{2N-1} p  F^{2,2N}\right)
+ \ns{ \dt^{2N} u}_{1} 
  \ls  \sqrt{\fe{2N}  } \fd{2N}
\end{equation}
and
\begin{equation} \label{tem en 2N00}
     \int_\Omega J   \dt^{2N-1} p  F^{2,2N}\ls (\fe{2N})^{3/2}.
\end{equation}
\end{prop}

\begin{proof}
This is just a restatement of Proposition 5.3 of \cite{GT_inf} in the time-differential form, which is proved by employing the geometric formulation \eqref{linear_geometrica1} and  using the estimates \eqref{p_F_e_01}--\eqref{p_F_e_02} of Lemma \ref{p_F_estimates}. Note that an integration by parts in time is used for the pressure term as  there is one more time derivative on $ p$ than can  be controlled. We refer to \cite{GT_inf} for more details.
\end{proof}

We then record a similar result for $N+2$  temporal derivatives.
\begin{prop}\label{i_temporal_evolution N+2}
It holds that
\begin{equation} \label{tem en N+2}
   \frac{d}{dt} \left(\ns{ \dt^{N+2} u}_{0} +  \as{\dt^{N+2}\eta}_{0}+2\int_\Omega J   \dt^{N+1} p  F^{2,N+2}\right)
+ \ns{ \dt^{N+2} u}_{1}
  \ls \sqrt{\fe{2N}   } \sdnn 
\end{equation}
and
\begin{equation} \label{tem en N+200}
     \int_\Omega J   \dt^{N+1} p  F^{2,N+2}\ls \sqrt{\fe{2N}   }\senn   .
\end{equation}
\end{prop}

\begin{proof}
This is just a restatement of Proposition 5.4 of \cite{GT_inf} when $m=2$ therein. The proof is similar to Proposition \ref{i_temporal_evolution 2N} except using \eqref{p_F_e_h_01}--\eqref{p_F_e_h_02} in place of \eqref{p_F_e_01}--\eqref{p_F_e_02}.
\end{proof}

\subsection{Energy evolution of Alinhac good unknowns}
To derive the energy evolution of the highest order horizontal spatial derivatives of the solution to \eqref{geometric}, as explained in Section \ref{strategy sec},  even using the geometric formulation \eqref{linear_geometrica1} will lead to the appearance of $\mathcal{K}\f$. To avoid this, we will appeal to  the reformulation by using Alinhac good unknowns. Recall that when  applying the differential operator  $\p^\al$ for $\alpha\in \mathbb{N}^{d-1}$ to \eqref{geometric}, we need to commute $\p^\alpha$   with each differential operator $\D$ in \eqref{geometric}. It is thus useful to establish the following general expressions.
For $i=  1, \, 2,\, d$, set
\beq\label{com1}
\p^\alpha \D_{i}f=\D_{i} \p^\alpha f-\D_{d}f \D_{i} \p^\alpha \varphi+\mathcal{C}^\alpha_{i}(f),
\eeq
where the commutator $\mathcal{C}^\alpha_{i}(f)$ is given for $i \neq d$ by
\beq\label{Cialpha}
\mathcal{C}^\alpha_{i}(f)= \mathcal{C}^\alpha_{i, 1}(f)+ \mathcal{C}^\alpha_{i,2}(f)
\eeq
with
\begin{align}
\label{Cialpha1} \mathcal{C}^{\alpha}_{i,1}&=-\[ \p^\alpha, \frac{\partial_{i} \varphi }{ \partial_{d} \phi }, \partial_{d} f  \], \\
\label{Cialpha2} \mathcal{C}^{\alpha}_{i,2}&=-\partial_{d} f \[ \p^\alpha, \partial_{i} \varphi, \frac{1}  {\partial_{d} \phi }\]     -
\partial_{i} \varphi\partial_zf\[\p^{\alpha-\alpha'},   \frac{1}{(\partial_{d} \phi )^2} \]\p^{\alpha'}\partial_{d} \eta
\end{align}
for any $\al'\le\al$ with $|\alpha'|=1$. Note that for $i= 1, \, d-1$ ,  $\partial_{i}\phi=\partial_{i}\varphi$ and that for $\alpha \neq 0$,
$\p^\alpha\partial_{d}\phi= \p^\alpha \partial_{d}\varphi$. For $i=d$, similar decomposition for the commutator holds (basically, it suffices to
replace $\partial_{i}\varphi$ by $1$ in the above expressions). Since $\D_{i}$ and $\D_{j}$ commute, it holds that
\beq\label{com122}
\p^\alpha \D_{i} f =\D_{i}(\p^\alpha f-\D_{d}f  \p^\alpha \varphi)+\D_{d}\D_{i}f \p^\alpha \varphi+\mathcal{C}^\alpha_{i}(f).
\eeq
It was first observed by Alinhac \cite{Alinhac} that the highest order term of $\varphi$ and hence $\eta$
will be canceled when we use the good unknown $\p^\alpha f-\D_{d}f\p^\alpha \varphi$.

We shall now derive the equations satisfied by the good unknowns 
\beq\label{gooddef}
U^\alpha: = \p^\alpha u - \D_{d}u\, \p^\alpha \eta  \text{ and }P^\al:=\p^\alpha p - \D_{d}p\, \p^\alpha \eta \text{ for  }\alpha\in \mathbb{N}^{d-1}\text{ with }\abs{\al}=4N.
\eeq
  \begin{lem}
  \label{lemValpha}
  For $  | \alpha | = 4N$, it holds that
   \beq\label{linear_geometrica}
   \begin{cases}    \D_{t} U^\alpha + u \cdot \nabla_\a  U^\alpha + \diva S_\a ( P^\al ,
  U^\alpha)=Q^{1,\al}& \text{\rm in }\Omega
      \\
      \nabla_\a  \cdot U^\alpha =Q^{2,\al}& \text{\rm in }\Omega
      \\   S_\a ( P^\al ,
  U^\alpha)\n =     \p^\alpha \eta  \n+Q^{3,\al} & \text{\rm on }\Sigma
  \\   \partial_{t} (\p^\alpha \eta)=U^\alpha \cdot \n +Q^{4,\al}& \text{\rm on }\Sigma
  \\U^\alpha =0 & \text{\rm on }\Sigma_b,
 \end{cases}
       \eeq
       where $Q^{i,\al},\ i=1,2,3,4$ are defined by \eqref{Q1def}, \eqref{Q2def}, \eqref{Q3def} and \eqref{Q4def}, respectively.
  \end{lem}
  \begin{proof}
We first prove the first equation in \eqref{lemValpha}.  Note that
   \beq
   \label{transportW} \D_{t} + u \cdot \nabla_\a
   =  \partial_{t}+ u_h \cdott D u +  U_{d} \partial_{d},\eeq
 where $U_{d}:= \frac{1}{ \partial_{d} \phi } (u\cdot\n- \partial_{t} \varphi) $
with $\n$ extended to $\Omega$ by $\n=(-D\varphi,1)$.
We can thus get  
\begin{align}  \label{T1}    \p^\alpha \( \D_{t} + u \cdot \nabla_\a    \)u
&    =   \(\partial_{t}+ u_h\cdot D +  U_{d} \partial_{d} \) \p^\alpha u + \big(u\cdot \p^\alpha \n-
\partial_{t} \p^\alpha \varphi \big) \D_{d} u
\\\nonumber&\quad- \D_{d} \p^\alpha \varphi \big( u \cdot \n - \partial_{t} \varphi\big)
 \D_{d}u
+ \mathcal{C}^\alpha(\mathcal{T}) \\
\nonumber & =  \big( \D_{t}+ u \cdot \nabla_\a \big)\p^\alpha u - \D_{d}u \big( \D_{t}+ u \cdot \nabla_\a ) \p^\alpha \varphi
 +  \mathcal{C}^\alpha(\mathcal{T})
 \\
\nonumber & =  \big( \D_{t}+ u \cdot \nabla_\a \big)U^\alpha + \D_{d} \( \D_{t}+ u \cdot \nabla_\a \)u \p^\alpha \varphi
-\D_{d} u \cdot \nabla_\a  u \p^\alpha \varphi
 +  \mathcal{C}^\alpha(\mathcal{T}),
 \end{align}
 where the commutator $\mathcal{C}^\alpha(\mathcal{T})$ is defined by
 \begin{align}\label{calphat}
 \mathcal{C}^\alpha(\mathcal{T})= &  \[\p^\alpha, u_h\] \cdot D u+ \[\p^\alpha, U_{d},\partial_{d}u\]+
   \[\p^\alpha, U_{d}, \frac{1 }{ \partial_{d} \phi }\] \partial_{d} u+ \frac{1 }{ \partial_{d}\varphi} \[\p^\alpha, u
   \]\cdot\n\partial_{d}u,\nonumber\\
       & +\p_d \varphi U_d \partial_d  u\[\p^{\alpha-\alpha'}, \frac{ 1 }{ (\partial_{d} \phi )^2}\]\p^{\alpha'}\partial_{d} \varphi
 \end{align}
for any $\al'\le\al$ with $|\alpha'|=1$. On the other hand, \eqref{com122}
implies that
\beq \label{TT1}
\p^\alpha \naba p=\nabla_\a   P^\alpha + \partial_{d}^\varphi \nabla_\a   p \p^\alpha \eta+\mathcal{C}^\alpha (p).
\eeq

  It remains to  compute $  \p^\alpha \Delta_\a u= \p^\alpha \nabla_\a  \cdot  (\sg_\a u).$ Note that
\beq \p^\alpha \nabla_\a  \cdot  (\sg_\a u) =  \nabla_\a  \cdot \big( \p^\alpha \, \sg_\a u \big) - \big( \D_{d}\, \sg_\a
   u \big) \nabla_\a
     ( \p^\alpha \varphi) + \mathcal{D^\alpha}\big( \sg_\a u \big)
     \eeq
     with $ \mathcal{D}^\alpha\big( \sg_\a u \big)_{i}=   \mathcal{C}_{j}^\alpha \big( \sg_\a u)_{ij},$ and
    \beq
    \label{comS}\p^\alpha  \big( \sg_\a u \big) =    \sg_\a \big( \p^\alpha u \big) - \D_{d} u \otimes \nabla_\a  \p^\alpha \varphi
     - \nabla_\a  \p^\alpha \varphi
  \otimes \D_{d} u  + \mathcal{E}^\alpha(u)\eeq
     with $\mathcal{E}^\alpha (u)_{ij}= \mathcal{C}^\alpha_{i}(u_{j})+ \mathcal{C}^\alpha_{j}(u_{i}).$
Then we deduce  that
     \begin{align}
     \label{deltaexp}
        \p^\alpha \Delta_\a u  & =  \nabla_\a  \cdot  \sg_\a(\p^\alpha u )  - \nabla_\a  \cdot \Big(  \D_{d}
       v \otimes \nabla_\a  \p^\alpha \varphi
     - \nabla_\a  \p^\alpha \varphi  \otimes \D_{d} u \Big)\nonumber\\
          &\quad -
       \big( \D_{d}\, \sg_\a u \big) \nabla_\a
     ( \p^\alpha \varphi)  +    \mathcal{D^\alpha}\big( \sg_\a u \big) +  \nabla_\a  \cdot  \mathcal{E}^\alpha(u)\nonumber
     \\
          & =\Delta_\a U^\alpha -
       \D_{d} \Delta_\a  u
      \p^\alpha \varphi   +    \mathcal{D^\alpha}\big( \sg_\a u \big) +  \nabla_\a  \cdot  \mathcal{E}^\alpha(
      u) .\end{align}
Hence, \eqref{T1}, \eqref{TT1}, \eqref{deltaexp} and the first equation in \eqref{geometric}  imply the first equation in \eqref{lemValpha} with $Q^{1,\al}$ defined by
\beq\label{Q1def}
Q^{1,\al}= \widetilde Q^{1,\al}+
        \nabla_\a  \cdot   \mathcal{E}^\alpha (u) + \mathcal{D^\alpha}\big( \sg_\a u \big),
\eeq
where
\beq
 \widetilde Q^{1,\al}=\D_{d}u \cdot \nabla_\a  u
        \p^\alpha \varphi-\mathcal{C}^\alpha(\mathcal{T}) -\mathcal{C}^\alpha(p) .
        \eeq

Next,  \eqref{com122} yields that
\beq
\label{div1}
\p^\alpha  \nabla_\a  \cdot u  = \nabla_\a \cdot U^\alpha + \partial_{d}^\varphi \nabla_\a \cdot u \p^\alpha \varphi
 -Q^{2,\al},
\eeq
where
\beq\label{Q2def}
Q^{2,\al}:=-  \sum_{i=1}^3 \mathcal{C}^\alpha_{i}(u_i).
\eeq
\eqref{div1} and the second equation in \eqref{geometric} imply the second equation in  \eqref{lemValpha}.

 Now applying $\p^\alpha$ to the third equation in \eqref{geometric}  and using \eqref{comS}, we get \begin{align}
 \label{bordV1}
  \nonumber&  \( \sg_\a \( \p^\alpha u \)  -  \D_{d} u \otimes \nabla_\a  \p^\alpha \varphi
     -  \nabla_\a  \p^\alpha u \otimes \D_{d} u   +   \mathcal{E}^\alpha(u) \) \n- \( \p^\alpha p -   \p^\alpha \eta \) \n
     \\&\quad =-
     \(   \sg_\a u - (p- \eta  )I_d \) \p^\alpha \n - \[\p^\alpha,  \sg_\a u-(p-\eta)I_d, \n\].\nonumber
     \\&\quad =-  \(  \sg_\a u - \sg_\a u  \n\cdot \n I_d \) \p^\alpha \n -   \[\p^\alpha, \sg_\a u-\sg_\a u  \n\cdot \n I_d,
     \n\]
     \\&\quad\equiv- \sg_\a u \Pi  \p^\alpha \n -    \[\p^\alpha,   \sg_\a u \Pi  , \n\],\nonumber\end{align}
where $\Pi=I_d-\n\otimes\n$. This yields the third equation in  \eqref{lemValpha} with $Q^{3,\al} $ defined by
   \beq\label{Q3def}
Q^{3,\al}: =- \D_{d}p\, \p^\alpha \eta -   \sg_\a u \Pi \p^\alpha \n  - \p^\alpha \eta \D_{d}\big( \sg_\a u \big)\n-   \[\p^\alpha,   \sg_\a u \Pi  , \n\]- \mathcal{E}^\alpha(u)  \n.
  \eeq

We then apply $\p^\alpha$ to the fourth equation in \eqref{geometric} to find
\beq
  \partial_{t} \p^\alpha \eta + u_h \cdott D  \p^\alpha \eta -\p^\alpha u \cdot \n=   \[ \p^\alpha, u, D\eta\] . \eeq
This yields the fourth equation in  \eqref{lemValpha} with $Q^{4,\al}$ defined by
\beq\label{Q4def}
Q^{4,\al}:=- u_h \cdott D  \p^\alpha \eta  -\D_d u\cdot \n\p^\alpha \eta+\widetilde Q^{4,\al},
  \eeq
  where
  \beq
\widetilde Q^{4,\al}=\[ \p^\alpha, u_h\cdot, D\eta\] .
  \eeq

  Finally, the fifth equation in \eqref{geometric}  follows directly since $\varphi=0$ on $\Sigma_b$.
\end{proof}

We shall present the estimates of  some of these nonlinear terms in \eqref{linear_geometrica}.

\begin{lem}\label{p_Q_estimates}
It holds that
\begin{align}\label{QQes1}
 \norm{\widetilde Q^{1,\al}}_0^2 +\ns{Q^{2,\al}}_0\ls \senn \sdm{2N} ,
\end{align}
\begin{align}\label{QQes2}
 \ns{ \mathcal{E}^\alpha (u) }_0+\as{ \mathcal{E}^\alpha (u) }_{-1/2} \ls \senn \sdm{2N} ,
\end{align}
\beq\label{QQes3}
\as{Q^{3,\al}}_{-1/2}\ls (\senn+\k)\sdm{2N}+\bar\k\f
\eeq
and
\beq\label{QQes4}
\as{\widetilde Q^{4,\al}}_{0}\ls \senn  \sdm{2N}+  \bar \k\sem{2N}.
\eeq
\end{lem}
\begin{proof}
The estimates \eqref{QQes1}--\eqref{QQes2} follow similarly as Lemma \ref{p_G_estimates_N+2}. For  the estimate \eqref{QQes3}, the additional terms result from the followings:
   \beq 
\as{   \sg_\a u \Pi \p^\alpha \n }_{-1/2} \ls \as{   \nabla u   }_{C^1(\Sigma)}\as{ \p^\alpha D\eta }_{-1/2}\ls\bar \k \as{ \p^\alpha \eta }_{1/2} \ls\bar \k\f \eeq
and
   \beq 
\as{   p^\alpha \eta \D_{d}\big( \sg_\a u \big)\n}_{-1/2}\ls \as{   \p^\alpha \eta }_{-1/2}\(\as{\nabla u}_{C^1(\Sigma)}+\as{\nabla u}_{C^2(\Sigma)}\)  \ls\sdm{2N} \k. \eeq
While for the estimate  \eqref{QQes4}, the additional term is due to that for  $|\al'|=1$:
\beq
\as{\p^{\al'} u_h \cdott\p^{\al-\al'}D \eta}_0\ls \as{D u_h}_{C^1(\Sigma)} \as{\p^{4N}\eta}_0\ls \bar \k\sem{2N}.
\eeq
The estimates \eqref{QQes3}--\eqref{QQes4} then follow. 
\end{proof}

We now present the energy evolution for $4N$ horizontal spatial derivatives. We compactly write
\beq\label{notaU}
\ns{U^{4N}}:=\sum_{\al\in \mathbb{N}^{d-1},|\al|=4N}\ns{U^{\al}}+\ns{u}.
\eeq
\begin{prop}\label{i_h_evolution 2N}
It holds that
\begin{equation}\label{i_te_0}
 \dt \(\ns{  U^{4N}}_{0} + \sem{2N} \)+   \ns{ U^{4N}}_{1}  \ls  \sqrt{\senn+\bar\k} \sem{2N}+ (\sqrt{\senn}+\k)\sdm{2N}+ \bar\k\f
.
\end{equation}
\end{prop}
\begin{proof}
Let $\al\in \mathbb{N}^{2}$ with $\abs{\al}=4N$. Taking the dot product of the first equation in $\eqref{linear_geometrica}$ with $U^\al$ and then integrating by parts, using the other equations in $\eqref{linear_geometrica}$ as Proposition \ref{p_upper_evolution 2N}, we obtain
\begin{align} \label{identity1a}
&\hal  \frac{d}{dt} \left(\int_\Omega  J \abs{U^\alpha}^2+\int_\Sigma \abs{\p^\al \eta}^2 \right)
+ \hal \int_\Omega J \abs{ \sg_{\mathcal{A}} U^\alpha}^2\nonumber
\\&\quad= \int_\Omega J \(   U^\alpha \cdot Q^{1,\al}+ P^\al  Q^{2,\al}\)
+\int_\Sigma -  U^\alpha \cdot Q^{3,\al}+ \p^\al \eta Q^{4,\al}.
\end{align}

We now estimate the right hand side of \eqref{identity1a}.  For the $Q^{1,\al}$ term, by \eqref{QQes1}, we   have
\begin{align}\label{qqqq1}
 \int_{\Omega}JU^\al\cdot\tQ^{1,\al}
\ls\norm{U^\al}_0\norm{\tQ^{1,\al}}_0\ls \norm{U^\al}_0\sqrt{\senn \sdm{2N}}.
\end{align}
It follows from the integration by parts and the trace theory that, by \eqref{QQes2},
\begin{align}\label{qqqq2}
 \int_{\Omega}  J U^\alpha\cdot\diva  \mathcal{E}^\alpha(u)    &=\int_\Sigma   U^\alpha \cdot \mathcal{E}^\alpha(u)\n
- \int_{\Omega}  J \naba
U^\alpha:
\mathcal{E}^\alpha(u) \nonumber\\\nonumber& \ls  \abs{  U^\alpha}_{1/2}\abs{\mathcal{E}^\alpha(u) }_{-1/2}+\norm{\nabla U^\alpha}_0\norm{\mathcal{E}^\alpha(u)}_0
\\& \ls
   \norm{  U^\alpha}_1\sqrt{\senn \sdm{2N}}.
\end{align}
By the definition \eqref{Cialpha},
\begin{align}\label{RSi}
 \int_{\Omega} J\   U^\alpha \cdot  \mathcal{D^\alpha}\big( \sg_\a u \big)   =  \int_{\Omega} J \mathcal C^\alpha_{j, 1}(\sg_\a u)_{ij}  U^\alpha_{j}  +   \int_{\Omega} J \mathcal C^\alpha_{j,
2}(\sg_\a u)_{ij}  U^\alpha_{j}  .
 \end{align}
For the first term, by expanding, it suffices to estimate terms like
\beq
 \int_{\Omega} J U^\alpha_{j} \p^\beta \big( \frac{\partial_{j} \varphi }{ \partial_{d} \phi }\big) \big(\p^{ {\gamma}} \partial_{d}(\sg_\a
v)_{ij} \big),
\eeq
where  $\beta \neq 0, \, {\gamma} \neq 0$ and $|\beta | + |{\gamma}|=4N .$ If $|\gamma| = 4N-1$ and hence $|\beta|=1$, we integrate by parts to have
\begin{align}
 \int_{\Omega} J U^\alpha_{j} \p^\beta \big( \frac{\partial_{j} \varphi }{ \partial_{d} \phi }\big) \big(\p^{ {\gamma}} \partial_{d}(\sg_\a
v)_{ij} \big)&=- \int_{\Omega}\p^{\gamma'}\(JU^\alpha_{j} \p^\beta \big( \frac{\partial_{j} \varphi }{ \partial_{d} \phi }\big)\)  \big(\p^{ {\gamma-\gamma'}} \partial_{d}(\sg_\a
v)_{ij} \big) \nonumber
\\&\ls\norm{U^\alpha}_1\sqrt{\senn  }\sqrt{  \sdm{2N}}.
\end{align}
For $|\gamma| \le 4N-2$ and hence $|\beta|\ge 2$,
\beq
 \int_{\Omega} J U^\alpha_{j} \p^\beta \big( \frac{\partial_{j} \varphi }{ \partial_{d} \phi }\big) \big(\p^{ {\gamma}} \partial_{d}(\sg_\a
v)_{ij} \big)\ls\norm{U^\alpha}_0\sqrt{\senn     \sdm{2N}}.
\eeq
Similarly, we have
\begin{align} \label{RSi2}
 \int_{\Omega} J \mathcal C^\alpha_{j,
2}(\sg_\a u)_{ij}  U^\alpha_{j}\ls\norm{U^\alpha}_0\sqrt{\senn     \sdm{2N}}.
 \end{align}
 It then follows from \eqref{RSi}--\eqref{RSi2} that
\begin{align}\label{qqqq3}
\int_{\Omega} J\   U^\alpha \cdot  \mathcal{D^\alpha}\big( \sg_\a u \big)  \ls  \norm{U^\alpha}_1\sqrt{\senn     \sdm{2N}}.
\end{align}
Hence, by \eqref{qqqq1},  \eqref{qqqq2} and  \eqref{qqqq3}, we get
\begin{align}\label{QQ1}
 \int_{\Omega}J    U^\alpha \cdot Q^{1,\al}
\ls \norm{ U^\alpha}_1\sqrt{\senn \sdm{2N}}.
\end{align}

For the $Q^{2,\al}$ term, by \eqref{QQes1}, we have
\begin{align} \label{1231}
&  \int_\Omega J   P^\al  Q^{2,\al}
\ls\norm{P^\al}_0\norm{Q^{2,\al}}_0\ls \sqrt{\sdm{2N}}\sqrt{\senn \sdm{2N}}.
\end{align}
For the $Q^{3,\al}$ term, by \eqref{QQes3} and the trace theory, we obtain
 \begin{align} \label{1235}
 \int_\Sigma -  U^\alpha \cdot Q^{3,\al} \le \abs{U^\alpha}_{1/2}\abs{ Q^{3,\al} }_{-1/2}\ls\norm{U^\alpha}_{1}\sqrt{(\senn+\k)\sdm{2N}+\bar\k\f}.
\end{align}
For the $Q^{4,\al}$ term,
the integration by parts gives
\begin{align}
  \int_{\Sigma }     \p^\alpha \eta \(  u_h\cdott D\p^\alpha \eta +\D_d  u \cdot \n
 \p^\alpha \eta\)
& =  \int_{\Sigma }     \(  -\hal D\cdott u_h   +\D_d  u \cdot \n
 \) \abs{ \p^\alpha \eta}^2 \nonumber
 \\&\ls   \abs{\nabla u}_{C^1(\Sigma)} \as{\p^\alpha \eta}_0\ls  \sqrt{\bar\k} \sem{2N}.
\end{align}
By \eqref{QQes4}, we have
\beq\label{QQ10}
 \int_{\Sigma }    \p^\alpha \eta   \tQ^{4,\al}
 \ls  \abs{\p^\alpha \eta}_0\abs{  \tQ^{4,\al} }_0 \ls  \sqrt{\sem{2N}}  \sqrt{\senn  \sdm{2N}+ \bar\k {\sem{2N}}}.
\eeq

Finally,   we may follow the estimates (5-22)--(5-25) of \cite{GT_inf} to have
\begin{equation}\label{QQ11}
\int_\Omega J \abs{ \sg_{\mathcal{A}} U^\al}^2\ge  \ns{ \sg U^\al}_{0}.
\end{equation}
Summing \eqref{QQ1}--\eqref{QQ11},
we deduce \eqref{i_te_0} from \eqref{identity1a} and \eqref{intro0} by using Cauchy's and Korn's inequalities.
\end{proof}

\section{Comparison results}
In this section we show that, up to some errors, the full energies and dissipations are comparable to those tangential ones at both $2N$ and $N+2$ levels. 

\subsection{Energy comparison} We begin with the result for the instantaneous energies. Recall the definitions of $\fe{2N}$,    $\senn$,  $\feb{2N}$, $\seb{N+2,2} $,  $\sem{2N}$ and $\k$ from  \eqref{p_energy_def1}, \eqref{i_energy_min_2}, \eqref{i_horizontal_energy},  \eqref{i_horizontal_energy1},  \eqref{pdd_dissipation_def1} and \eqref{kdef}, respectively.

\begin{thm}\label{eth}
It holds that
\begin{equation}\label{e2n1}
 \fe{2N} \lesssim  \feb{2N}  + (\fe{2N})^2 +\k\sem{2N}
\end{equation}
and
\begin{equation}\label{en+2}
\senn \lesssim  \seb{N+2,2} +  \fe{2N} \senn   .
\end{equation}
\end{thm}
\begin{proof}
We first prove \eqref{e2n1}.
Note that the definition \eqref{i_horizontal_energy} of $\feb{2N}$ guarantees that
\begin{equation}\label{p_E_b_4}
  \ns{\dt^{2N} u}_{0}+ \as{\eta}_{4N-1}+\sum_{j=1}^{2N} \as{\dt^j \eta}_{4N-2j}  \ls   \feb{2N} .
\end{equation}
We let $j=0,\dots,2N-1$ and then apply $\partial_t^j$ to the equations in \eqref{linear_perturbed} to find
\begin{equation}\label{jellip}
\begin{cases}
- \Delta \partial_t^j u+\nabla\partial_t^j p=-\partial_t^{j+1} u+\partial_t^j G^1
 &\text{in }\Omega
\\ \diverge\partial_t^j u=\partial_t^j G^2&\text{in }\Omega
\\   (\partial_t^j p I_d- \mathbb{D}(\partial_t^j u ) ) e_d = \partial_t^j\eta   e_d +\partial_t^jG^3&\text{on }\Sigma
\\ \partial_t^j u=0 &\text{on }\Sigma_b.
\end{cases}
\end{equation}
Applying the elliptic estimates of Lemma \ref{i_linear_elliptic} with $r=4N-2j\ge 2$ for $j=1,\dots,2N-1$ to the problem \eqref{jellip} and using \eqref{p_E_b_4} and \eqref{p_G_e_0}, we obtain
\begin{align}\label{p_E_b_2}
  \norm{\dt^j  u  }_{4N-2j}^2 + \norm{\dt^j  p  }_{4N-2j-1}^2
 &  \ls
\norm{\dt^{j+1} u   }_{4N-2j-2 }^2 + \norm{ \dt^j G^1   }_{4N-2j-2}^2
+ \norm{\dt^j  G^2  }_{4N-2j-1}^2\nonumber \\ &\quad
 + \abs{\dt^j  \eta }_{4N-2j-3/2}^2  + \abs{\dt^j  G^3  }_{4N-2j-3/2}^2\nonumber
 \\&  \ls   \norm{\dt^{j+1} u   }_{4N-2(j+1) }^2 + \feb{2N} + (\fe{2N})^2 .
\end{align}
A simple induction on \eqref{p_E_b_2} yields, by \eqref{p_E_b_4} again,
\begin{equation}\label{claim22}
\sum_{j=1}^{2N } \norm{\dt^{j}  u  }_{2N-2j }^2 + \sum_{j=1}^{4N-1}\norm{\dt^{j} p  }_{4N-2j-1}^2 \ls \ns{\dt^{2N} u}_{0}+
\feb{2N}   +(\fe{2N})^2 \ls
\feb{2N}   + (\fe{2N})^2 .
\end{equation}
On the other hand, applying the elliptic estimates of Lemma \ref{i_linear_elliptic} with $r=4N-1$ to the problem \eqref{jellip} for $j=0$  and using \eqref{claim22}, \eqref{p_E_b_4} and \eqref{p_G_e_1}, we have
\begin{align}\label{p_E_b}
  \norm{ u  }_{4N }^2 + \norm{p  }_{4N-1}^2
 &  \ls
\norm{\dt  u   }_{4N -2 }^2 + \norm{   G^1   }_{4N -2}^2
+ \norm{  G^2  }_{4N-1}^2
 + \abs{   \eta }_{4N-3/2}^2  + \abs{G^3  }_{4N-3/2}^2\nonumber
 \\&  \ls    \feb{2N}   + (\fe{2N})^2+\k\sem{2N}.
\end{align}
Consequently, by the definition \eqref{p_energy_def1} of $\fe{2N}$,  summing \eqref{claim22} and \eqref{p_E_b} gives \eqref{e2n1}.

Now we prove \eqref{en+2}. Note that the definition \eqref{i_horizontal_energy1} of $\seb{N+2,2}$ guarantees that
\begin{equation}\label{p_E_b_4'}
  \ns{\dt^{N+2} u}_{0}+\as{D^2 \eta}_{2(N+2)-2} +\sum_{j=1}^{N+2}  \as{\dt^j \eta}_{2(N+2)-2j}  \ls   \seb{N+2,2} .
\end{equation}
It follows similarly as the derivation of  \eqref{claim22} and \eqref{p_E_b}, except using \eqref{p_E_b_4'} in place of \eqref{p_E_b_4} and \eqref{p_G_e_h_0} in place of \eqref{p_G_e_0}--\eqref{p_G_e_1}, that
\begin{align}\label{claim222}
 &\norm{ D^2 u  }_{2(N+2)-2 }^2+\norm{D^2 p  }_{2(N+2)-3}^2+\sum_{j=1}^{N+2 } \norm{\dt^{j}  u  }_{2(N+2)-2j }^2 + \sum_{j=1}^{N+1}\norm{\dt^{j} p  }_{4N-2j-1}^2 \nonumber
 \\&\quad\ls \seb{N+2,2}   +\ns{ \bar{D}^2\bar{\nab}_0^{2(N+2)-4} G^1}_{0} +  \ns{ \bar{D}^2\bar{\nab}_0^{2(N+2)-4}  G^2}_{1}  +
 \as{\bar{D}^2 \bar{D}_{0}^{2(N+2)-4} G^3}_{1/2}
 \nonumber
 \\&\quad\ls
\seb{N+2,2}   + \fe{2N} \senn .
\end{align}
By an inspection of the definition \eqref{i_energy_min_2} of $\senn$ and the estimate \eqref{claim222}, it remains to improve the estimates of terms without temporal derivatives in \eqref{claim222}. First,
by   Poincar\'e's inequality of Lemma \ref{poincare_usual} and the second equation in \eqref{linear_perturbed}, using   \eqref{claim222} and \eqref{p_G_e_h_0}, we have
\begin{align}\label{p_E_b_`114}
& \norm{D  u_d   }_{2(N+2)-1}^2 \ls    \norm{D^2 u_d   }_{2(N+2)-2}^2+\norm{\p_d D  u_d   }_{2(N+2)-2}^2
 \nonumber\\  &\quad
 \ls    \norm{D^2 u }_{2(N+2)-2}^2+\norm{ D  G^2  }_{2(N+2)-2}^2
  \ls \seb{N+2,2}   +  \fe{2N} \senn  .
\end{align}
Then by the vertical component of the first equation in \eqref{linear_perturbed}, using  \eqref{p_E_b_`114}, \eqref{claim222} and \eqref{p_G_e_h_0}, we obtain
\begin{align}\label{p_E1_b_`12}
 \norm{D \p_d   p  }_{2(N+2)-3}^2 & \ls
 \norm{D u_d  }_{2(N+2)-1}^2 +\norm{D\p_t u_d  }_{2(N+2)-3}^2+\norm{D  G^1_d }_{2(N+2)-3}^2
 \nonumber\\ &   \ls \seb{N+2,2}   +  \fe{2N} \senn .
\end{align}
Next, by the horizontal component of the first equation in \eqref{linear_perturbed}, using   \eqref{claim222}, \eqref{p_E1_b_`12} and \eqref{p_G_e_h_0}, we get
\begin{align}\label{p_E_b_`1}
 \norm{\nabla\p_d ^2 u_h  }_{2(N+2)-3}^2 & \ls \norm{  \nabla D^2 u_h  }_{2(N+2)-3}^2+\norm{\nabla\dt u_h }_{2(N+2)-3}^2+\norm{\nabla D p  }_{2(N+2)-3}^2
  \nonumber\\ & \quad+\norm{\nabla G_h^1}_{2(N+2)-3}^2
 \nonumber\\ &   \ls \seb{N+2,2}   +  \fe{2N} \senn .
\end{align}
By Poincar\'e's inequality of Lemma \ref{poincare_b} and the horizontal component of the fourth equation in \eqref{linear_perturbed}, using \eqref{claim222}, \eqref{p_E_b_`1}, the  trace theory and \eqref{p_G_e_h_0}, we have
\begin{align}\label{p_E_b_`112}
  \norm{\p_d D  u_h  }_{2(N+2)-2}^2&\ls \as{D\p_d   u_h  }_{0}+\norm{\nabla\p_d D  u_h  }_{2(N+2)-3}^2\nonumber
\\&\ls \as{ D^2 u_d }_{0}+\as{ D G^3}_{0}+\norm{\nabla\p_d D  u_h  }_{2(N+2)-3}^2\nonumber
\\&\ls \ns{ D^2 u_d }_{2(N+2)-2}+\norm{\p_d ^2D  u_h  }_{2(N+2)-3}^2+\as{ D G^3}_{0}\nonumber
\\&\ls \seb{N+2,2}   +  \fe{2N} \senn .
\end{align}
Then by Poincar\'e's inequality of Lemma \ref{poincare_usual}, using \eqref{claim222} and \eqref{p_E_b_`112}, we obtain
\begin{align}\label{p_E_b_`11}
\norm{D  u_h  }_{2(N+2)-1}^2 \ls  \norm{\nabla D  u_h  }_{2(N+2)-2}^2
   \ls \seb{N+2,2}   +  \fe{2N} \senn .
\end{align}
Now, by the second equation in \eqref{linear_perturbed} again, using  \eqref{p_E_b_`11}, \eqref{p_E_b_`114} and \eqref{p_G_e_h_0}, we have
\begin{align}\label{p_E_b_`115}
& \norm{  u_d   }_{2(N+2)}^2 \ls    \norm{D  u_d   }_{2(N+2)-1}^2+\norm{\p_d   u_d   }_{2(N+2)-1}^2
 \nonumber\\  &\quad
 \ls    \norm{D  u }_{2(N+2)-1}^2+\norm{    G^2  }_{2(N+2)-1}^2
  \ls    \seb{N+2,2}   + \fe{2N} \senn  .
\end{align}
By the vertical component of the first equation in \eqref{linear_perturbed}, using   \eqref{claim222}, \eqref{p_E_b_`115} and \eqref{p_G_e_h_0}, we obtain
\begin{align}\label{p_E_b_`3}
\norm{\p_d  p  }_{2(N+2)-2}^2&\ls \norm{\dt u_d  }_{2(N+2)-2}^2+\norm{u_d  }_{2(N+2)}^2+\norm{G^1_d   }_{2(N+2)-2}^2\nonumber
  \\&\ls \seb{N+2,2}   +  \fe{2N} \senn .
\end{align}
Consequently,  by the definition \eqref{i_energy_min_2} of $\senn$, summing \eqref{claim222}, \eqref{p_E_b_`1} and \eqref{p_E_b_`11}--\eqref{p_E_b_`3}  gives \eqref{en+2}.
\end{proof}

\subsection{Dissipation comparison}
Now we consider a similar result for the dissipations. Recall the definitions of $\fd{2N},\ \sdnn,\ \fdb{2N},\ \sdb{N+2,2}$, $\sdm{2N}$  and $\f$  from  \eqref{p_dissipation_def1}, \eqref{p_ldissipation_def}, \eqref{i_horizontal_dissipation},  \eqref{i_horizontal_dissipation1},  \eqref{pdd_dissipation_def2} and \eqref{pdd_dissipation_def3}, respectively. We also recall the notation \eqref{notaU} of  $\ns{ U^{4N} }_{1} $ for Alinhac good unknowns.

\begin{thm}\label{dth}
It holds that
\begin{equation}\label{d2n1}
 \fd{2N} \lesssim  \fdb{2N}  +\fe{2N} \fd{2N}+\k  \sdm{2N},
\end{equation}
\begin{equation}\label{d2n}
 \sdm{2N} \lesssim   \fdb{2N} +\ns{ U^{4N} }_{1} +\fe{2N} \fd{2N}+\k \f
\end{equation}
and
\begin{equation}\label{dn+2}
\sdnn  \lesssim  \sdb{N+2,2} +  \fe{2N} \sdnn  .
\end{equation}
\end{thm}

\begin{proof}
We first prove \eqref{d2n1}. Notice that we have not yet derived an estimate of $\eta$ in terms of the dissipation, so we can not apply the elliptic estimates of Lemma \ref{i_linear_elliptic} as in Theorem \ref{eth}.  It is crucial to observe that we can get higher regularity estimates of $u$ on the boundary $\Sigma$ from $\fdb{2N}$. Indeed, since $\Sigma$ is flat, we may use the definition of Sobolev norm on $\Sigma$ and the trace
theorem to obtain, by the definition \eqref{i_horizontal_dissipation} of $\fdb{2N}$, that for $j=1,\dots, 2N$,
\begin{align}\label{n21}
\as{\partial_t^{j} u}_{4N-2j+1/2}&  \lesssim \as{\partial_t^{j}{u}}_{0}
+\as{ D^{4N-2j}\partial_t^{j} u }_{1/2}
\lesssim \ns{ \partial_t^{j}  {u} }_{1} +\ns{ D ^{4N-2j}\partial_t^{j} {u} }_{1}
 \lesssim \fdb{2N}
 \end{align}
 and
 \begin{align}\label{n212}
\as{  u}_{4N-1/2}&  \lesssim \as{{u}}_{0}
+\as{ D^{4N-1} u }_{1/2}
\lesssim \ns{ {u} }_{1} +\ns{ D ^{4N-1}{u} }_{1}
 \lesssim \fdb{2N}
 \end{align}
 This motivates us to use the elliptic estimates of Lemma \ref{i_linear_elliptic2}.

Now we let $j=0,\dots,2N-1$ and then apply $\partial_t^j$ to the equations in \eqref{linear_perturbed} to find
\begin{equation}\label{jellip2}
\begin{cases}
- \Delta \partial_t^j u+\nabla\partial_t^j p=-\partial_t^{j+1} u+\partial_t^j G^1
 &\text{in }\Omega
\\ \diverge\partial_t^j u=\partial_t^j G^2&\text{in }\Omega
\\  \partial_t^j u=\partial_t^j u &\text{on }\Sigma
\\ \partial_t^j u=0 &\text{on }\Sigma_b.
\end{cases}
\end{equation}
Applying the elliptic estimates of Lemma \ref{i_linear_elliptic2} with $r=4N-2j+1\ge 3$  for $j=1,\dots,2N-1$ to the problem \eqref{jellip2}, using \eqref{n21} and \eqref{p_G_d_0}, we obtain
\begin{align}\label{p_D_b_0}
 &\norm{\dt^j  u  }_{4N-2j+1}^2 + \norm{\nab \dt^j  p  }_{4N-2j-1}^2\nonumber
 \\  &\quad\ls\norm{\dt^j  u  }_{0}^2 +
\norm{\dt^{j+1} u   }_{4N-2j-1 }^2 + \norm{ \dt^j G^1   }_{4N-2j-1}^2
+ \norm{\dt^j  G^2  }_{4N-2j}^2 + \abs{\dt^j  u  }_{4N-2j+1/2}^2\nonumber
\\&\quad\ls  \fdb{2N} + \norm{\dt^{j+1} u   }_{4N-2(j+1)+1 }^2 +  \fe{2N} \fd{2N}.
\end{align}
A simple induction on \eqref{p_D_b_0} yields
\begin{align}\label{p_D_b_1}
\sum_{j=1}^{2N} \norm{\dt^{j}  u  }_{4N-2j+1}^2 + \sum_{j=1}^{2N-1}\norm{\nab \dt^{j} p  }_{4N-2j-1}^2  &\ls \fdb{2N}   + \ns{\dt^{2N} u}_{1}+
\fe{2N} \fd{2N}
\nonumber\\&\ls
\fdb{2N}   +\fe{2N} \fd{2N}.
\end{align}
On the other hand, applying the elliptic estimates of Lemma \ref{i_linear_elliptic2} with $r=4N$ to the problem \eqref{jellip2} for $j=0$  and using \eqref{n212}, \eqref{p_D_b_1} and \eqref{p_G_d_1}, we have
\begin{align}\label{p_D_b_2}
 \norm{ u  }_{4N}^2 + \norm{\nab   p  }_{4N-2}^2
 & \ls\norm{   u  }_{0}^2 +
\norm{\dt  u   }_{4N-2 }^2 + \norm{  G^1   }_{4N-2}^2
+ \norm{  G^2  }_{4N -1}^2 + \abs{  u  }_{4N- 1/2}^2\nonumber
\\& \ls     \fdb{2N} +  \fe{2N} \fd{2N}+\k  \sdm{2N}.
\end{align}

Now  we estimate $\eta$, and we turn to the boundary conditions in \eqref{linear_perturbed}. For the $\eta$ term, i.e. without temporal derivatives, we use the vertical component of the third equation in \eqref{linear_perturbed}
\begin{equation}\label{pb1}
 \eta= p- 2\partial_d u_{d} -G_{d}^3\quad\text{on }\Sigma.
\end{equation}
Notice that at this point we do not have any bound of $p$ on the boundary $\Sigma$, but we have bounded  $\nabla p$ in $\Omega$.  Applying $ D $  to \eqref{pb1},  by the trace theory, \eqref{p_D_b_2} and \eqref{p_G_d_1}, we obtain
\begin{align}\label{n51}
  \as{  D  \eta}_{4N-5/2}
& \lesssim \as{   D  p }_{4N-5/2}  + \as{   D  \partial_d u_d  }_{4N-5/2} +\as{ D  G^3_d }_{ 4N-5/2 }
\nonumber
\\& \ls     \fdb{2N} +  \fe{2N} \fd{2N}+\k  \sdm{2N}.
\end{align}
For the term $\dt^j \eta$ for $j\ge 1$, we use instead  the fourth equation in \eqref{linear_perturbed} to gain the regularity:
\begin{equation}\label{n61}
\partial_t\eta=u_d +G^4\quad\text{on }\Sigma.
\end{equation}
Indeed, for $\partial_t\eta$, we   use \eqref{n61}, \eqref{n212} and \eqref{p_G_d_1} to find
\begin{align}\label{eta2}
 \as{\partial_t \eta}_{4N-3/2} &\lesssim \as{u_d }_{4N-3/2} +\as{ G^4}_{4N-3/2}
\lesssim  \fdb{2N}   +\fe{2N} \fd{2N}+\k  \sdm{2N}.
\end{align}
For $j=2,\dots,2N+1$ we apply $\partial_t^{j-1}$ to \eqref{n61} to see, by \eqref{n21} and \eqref{p_G_d_0}, that
\begin{align}\label{eta1}
\as{\partial_t^j\eta}_{4N-2j+5/2} & \le  \as{\partial_t^{j-1}u_d }_{4N-2j+5/2} +\as{\partial_t^{j-1}G^4}_{4N-2j+5/2}\nonumber
\\
&\lesssim \as{\partial_t^{j-1}u_d  }_{{4N-2(j-1)+1/2}} +\as{\partial_t^{j-1}G^4}_{4N-2(j-1)+1/2}\nonumber
\\& \lesssim \fdb{2N}   +\fe{2N} \fd{2N}.
\end{align}

The $\dt^j \eta$ estimates allows us to further bound $\ns{\partial_t^jp}_0$ for $j=1,\dots,2N-1$. Indeed, applying $\partial_t^j$ for $j=1,\dots,2N-1$ to \eqref{pb1} and employing the trace theory, by \eqref{eta2}--\eqref{eta1} and \eqref{p_G_d_0}, we find
\begin{align}\label{pp1}
\as{\partial_t^j p}_{0}
&\lesssim \as{\partial_t^j \eta}_{0}  + \as{\partial_d \partial_t^j u_d }_{0}  + \as{\partial_t^j G^3_d }_{0}\lesssim \as{\partial_t^j \eta}_{0}+  \ns{ \partial_t^j u }_{2}+\as{\partial_t^j G^3}_{0}\nonumber
\\&\lesssim  \fdb{2N}   +\fe{2N} \fd{2N}+\k  \sdm{2N}.
\end{align}
By Poincar\'e's inequality of Lemma \ref{poincare_b}, by \eqref{p_D_b_1} and \eqref{pp1}, we have
\begin{equation}\label{pp2}
\ns{\partial_t^j p}_{0}
\lesssim \as{\partial_t^j p}_{0} +\ns{\p_d\partial_t^j p}_{0} \ls\fdb{2N} +  \fe{2N} \fd{2N}+\k  \sdm{2N}.
\end{equation}
Consequently, by the definition \eqref{p_dissipation_def1} of $\fd{2N}$,  summing \eqref{p_D_b_1}, \eqref{p_D_b_2}, \eqref{n51}, \eqref{eta2}, \eqref{eta1} and \eqref{pp2} gives \eqref{d2n1}.

We next prove \eqref{d2n}.  We must resort to Alinhac good unknowns $U^\al$ for $\abs{\al}=4N$. Indeed, by the definition of $U^\al$ in \eqref{gooddef}, the trace theory and Lemma \ref{i_sobolev_product_2}, we have
 \begin{align}\label{n213}
\as{  u}_{4N+1/2}&  \lesssim \as{{u}}_{0}
+\as{ D^{4N} u }_{1/2} \lesssim \as{{u}}_{0}
+\as{U^{4N}}_{1/2} +\as{\p_d ^\a u   D^{4N}\eta}_{1/2} \nonumber
 \\&\lesssim \ns{ U^{4N} }_{1}  +\as{\p_d u  }_{C^1(\Sigma)} \as{D^{4N}\eta}_{1/2} \lesssim  \ns{ U^{4N} }_{1} +\bar\k \f.
 \end{align}
 Applying the elliptic estimates of Lemma \ref{i_linear_elliptic2} with $r=4N+1$ to the problem \eqref{jellip2} for $j=0$ and using \eqref{n213}, \eqref{p_D_b_1} and \eqref{p_G_d_2}, we obtain
\begin{align}\label{p_D_b_221}
 \norm{ u  }_{4N+1}^2 + \norm{\nab   p  }_{4N-1}^2
   & \ls\norm{   u  }_{0}^2 +
\norm{\dt  u   }_{4N-1 }^2 + \norm{  G^1   }_{4N-1}^2
+ \norm{  G^2  }_{4N }^2 + \abs{  u  }_{4N+ 1/2}^2\nonumber
\\& \ls     \fdb{2N} +\ns{ U^{4N} }_{1}+  \fe{2N}  \sdm{2N} +\k \f.
\end{align}
Applying $ D $  to \eqref{pb1},  by the trace theory, \eqref{p_D_b_221} and \eqref{p_G_d_2}, we have
\begin{align}\label{n512}
  \as{  D  \eta}_{4N-3/2}
& \lesssim \as{   D  p }_{4N-3/2}  + \as{   D  \partial_d u_d  }_{4N-3/2} +\as{ D  G^3_d }_{ 4N-3/2 }
\nonumber
\\& \ls      \fdb{2N} +\ns{ U^{4N} }_{1}+  \fe{2N}  \sdm{2N} +\k\f.
\end{align}
On the other hand, we use \eqref{n61}, \eqref{n213} and \eqref{p_G_d_2} to find
\begin{align}\label{eta22}
 \as{\partial_t \eta}_{4N-1/2} \lesssim \as{u_d }_{4N-1/2} +\as{ G^4}_{4N-1/2}\lesssim   \ns{ U^{4N} }_{1}  +\fe{2N} \fd{2N}+\k \f.
\end{align}
 Consequently,  by the definition \eqref{pdd_dissipation_def2} of $\sdm{2N}$, summing  \eqref{p_D_b_221}--\eqref{eta22} gives \eqref{d2n}.

Now we prove \eqref{dn+2}. Note that the definition \eqref{i_horizontal_dissipation1} of $\sdb{N+2,2}$ and the trace theory, similarly as \eqref{n21}--\eqref{n212}, guarantee that
\begin{align}\label{n21''}
\sum_{j=0}^{N+2}\as{\partial_t^{j} u}_{4N-2j+1/2}\lesssim \sdb{N+2,2}
 \end{align}
It follows similarly as the derivation of  \eqref{p_D_b_1}, \eqref{p_D_b_2}, \eqref{n51}, \eqref{eta1} and  \eqref{pp2}, except using \eqref{n21''} in place of \eqref{n21}--\eqref{n212} and \eqref{p_G_d_h_0} in place of \eqref{p_G_d_0}--\eqref{p_G_d_1}, that

\begin{align}\label{claim22267}
 &\norm{ D^2 u  }_{2(N+2)-1 }^2+\norm{\nabla D^2 p  }_{2(N+2)-3}^2+\sum_{j=1}^{N+2 } \norm{\dt^{j}  u  }_{2(N+2)-2j+1 }^2 + \norm{\nabla \dt p  }_{2(N+2)-2}^2 \nonumber
 \\&\quad+ \sum_{j=1}^{N+1}\norm{\dt^{j} p  }_{2(N+2)-2j}^2 +\as{  D^3  \eta}_{2(N+2)-7/2}+\sum_{j=2}^{N+3}\as{\partial_t^j\eta}_{2(N+2)-2j+5/2} \nonumber
 \\&\quad\ls \sdb{N+2,2}   +\ns{ \bar{D}^2\bar{\nab}_0^{2(N+2)-3} G^1}_{0} +  \ns{ \bar{D}^2\bar{\nab}_0^{2(N+2)-3}  G^2}_{1}+\as{ D^3  G^3_d }_{2(N+2)-7/2 }\nonumber
 \\&\qquad+  \as{ \bar{\nab}_0^{2(N+2)-2} \dt G^4}_{1}
 \nonumber
 \\&\quad\ls
\sdb{N+2,2}   + \fe{2N} \sdnn .
\end{align}
By an inspection of the definition  \eqref{p_ldissipation_def} of $\sdnn$ and the estimate \eqref{claim22267}, it remains to improve the estimates of $u$ and $ p$ terms without temporal derivatives in \eqref{claim22267} and derive the estimate of $\dt\eta$. First,
by the second equation in \eqref{linear_perturbed}, using   \eqref{claim22267} and \eqref{p_G_d_h_0}, we have
\begin{align}\label{p_E_b_`5}
  \norm{D u_d   }_{2(N+2)}^2  &   \ls \norm{D^2 u_d   }_{2(N+2)-1}^2+ \norm{ D  \p_d  u_d   }_{2(N+2)-1}^2
  \nonumber\\ &   \ls  \norm{ D^2 u  }_{2(N+2)-1}^2+  \norm{ D G^2  }_{2(N+2)-1}^2
   \ls \sdb{N+2,2}   +  \fe{2N} \sdnn .
\end{align}
We then use \eqref{n61}, \eqref{p_E_b_`5}, the trace theory and \eqref{p_G_d_h_0} to find
\begin{align}\label{eta2'}
 \as{D\partial_t \eta}_{2(N+2)-3/2} \lesssim \as{D u_d }_{2(N+2)-3/2} +\as{D G^4}_{2(N+2)-3/2}\lesssim  \sdb{N+2}   +\fe{2N} \sdnn .
\end{align}
By the vertical component of the first equation in \eqref{linear_perturbed}, using  \eqref{p_E_b_`5}, \eqref{claim22267} and \eqref{p_G_d_h_0}, we obtain
\begin{align}\label{p_D_b_`12}
 \norm{D \p_d   p  }_{2(N+2)-2}^2 & \ls
 \norm{D u_d  }_{2(N+2)}^2 +\norm{D\p_t u_d  }_{2(N+2)-2}^2+\norm{D  G^1_d }_{2(N+2)-2}^2
 \nonumber\\ &   \ls \sdb{N+2,2}   +  \fe{2N} \sdnn .
\end{align}
By the horizontal component of the first equation in \eqref{linear_perturbed}, using   \eqref{claim22267}, \eqref{p_D_b_`12} and \eqref{p_G_d_h_0}, we have
\begin{align}\label{p_D_b_`1}
 \norm{\nabla\p_d ^3 u_h  }_{2(N+2)-3}^2 & \ls \norm{  \nabla \p_d  D^2 u_h  }_{2(N+2)-3}^2+\norm{\nabla\p_d \dt u_h  }_{2(N+2)-3}^2+\norm{\nabla \p_d  D p  }_{2(N+2)-3}^2
  \nonumber\\&\quad+\norm{\nabla\p_d  G_h^1}_{2(N+2)-3}^2
 \nonumber\\ &   \ls \sdb{N+2,2}   +  \fe{2N} \sdnn .
\end{align}
Next, by the second equation in \eqref{linear_perturbed}, using   \eqref{p_D_b_`1} and \eqref{p_G_d_h_0}, we obtain
\begin{align}\label{p_D_b_`2}
  \norm{\p_d ^4 u_d   }_{2(N+2)-3}^2\ls \norm{    \p_d ^3 D u_h  }_{2(N+2)-3}^2+ \norm{ \p_d ^3 G^2}_{2(N+2)-3}^2
  \ls \sdb{N+2,2}   +  \fe{2N} \sdnn .
\end{align}
Finally, by the vertical component of the first equation in \eqref{linear_perturbed}, using   \eqref{claim22267}, \eqref{p_D_b_`2} and \eqref{p_G_d_h_0}, we have
\begin{align}\label{p_D_b_`3}
 \norm{  \p_d ^3p  }_{2(N+2)-3}^2 & \ls
 \norm{\p_d ^2\Delta u_d  }_{2(N+2)-3}^2 +\norm{  \p_d ^2 \p_t u_d  }_{2(N+2)-3}^2+\norm{ \p_d ^2 G_d ^1}_{2(N+2)-3}^2
 \nonumber\\ &   \ls \sdb{N+2,2}   +  \fe{2N} \sdnn .
\end{align}
Consequently, by the definition \eqref{p_ldissipation_def} of $\sdnn$, summing \eqref{claim22267}--\eqref{p_D_b_`3} gives \eqref{dn+2}.
\end{proof}

\section{A priori estimates}
\subsection{Bounded estimates of $\fe{2N}$ and $\fd{2N}$}
Now we show the boundedness of $\fe{2N} (t)+\int_0^t\fd{2N} $.

\begin{prop} \label{Dgle}
There exists $\delta>0$ so that if $\mathcal{G}_{2N} (T)\le\delta$, then
\begin{equation}\label{Dg}
\fe{2N} (t)+\int_0^t\fd{2N}(r)\,dr \lesssim
\fe{2N} (0) +  (\g(t))^2.
\end{equation}
\end{prop}

\begin{proof}
Note first that since $\fe{2N} (t)\le
{\mathcal{G}}_{2N} (T)\le \delta$, by taking $\delta$ small, we
obtain from \eqref{e2n1} of Theorem \ref{eth} and \eqref{d2n1} of
Theorem \ref{dth} that
\begin{equation}\label{q0}
\feb{2N}  \lesssim\fe{2N}
\lesssim \feb{2N}+ \k\sem{2N}  , \text{ and }\fdb{2N}  \lesssim\fd{2N}
\lesssim \fdb{2N} +\k\sdm{2N}  .
\end{equation}

Now we integrate \eqref{p_u_e_00} of Proposition  \ref{p_upper_evolution 2N} and \eqref{tem en 2N} of Proposition \ref{i_temporal_evolution 2N} in time and then add up the two resulting estimates to conclude that, by  the definition \eqref{i_horizontal_energy} of $\feb{2N}$,  the definition \eqref{i_horizontal_dissipation} of $\fdb{2N}$ and  \eqref{tem en 2N00},
\begin{align} \label{q20}
\feb{2N} (t)+\int_0^t \fdb{2N}   \lesssim  &
\fe{2N} (0) +
(\fe{2N} (t))^{ 3/2}  + \int_0^t  \( \sqrt{\fe{2N}}  \fd{2N}  +\sqrt{ \fd{2N} \k\sdm{2N} } \) .
\end{align}
By \eqref{q0}, we may improve \eqref{q20} to be, since $\delta$ is small and by Cauchy's inequality, \begin{align} \label{q20123}
\fe{2N} (t)+\int_0^t \fd{2N}   \lesssim  &
\fe{2N} (0) + \k(t)\sem{2N}(t)   + \int_0^t \k\sdm{2N}  .
\end{align}
By the estimate \eqref{kes} of Lemma \ref{klem} and the definition \eqref{ggdef} of $\g(t)$, we obtain
\beq\label{decayk}
\k(t)\ls \fe{2N}(t)^{\kappa_d/2} \senn(t)^{1-\kappa_d/2} \ls \g(t) (1+t)^{-2+\kappa_d},
\eeq
and hence we deduce from \eqref{q20123} that
\begin{align} \label{q2011}
\fe{2N} (t)+\int_0^t \fd{2N}(r)\,dr  & \lesssim
\fe{2N} (0) +   \g(t)  \frac{\sem{2N}(t)}{(1+t)^{2-\kappa_d}}
 +  \int_0^t \g(r)\frac{\sdm{2N}(r)}{(1+r)^{2-\kappa_d}} \,dr \nonumber
 \\&\ls \fe{2N} (0) +  (\g(t))^2
\end{align}
for $2-\kappa_d\ge \vartheta+\kappa_d> \vartheta>0$.
This gives \eqref{Dg}.
\end{proof}

\subsection{Growth estimates of $\sem{2N}$ and $\sdm{2N}$}
We first derive the estimates of $\sem{2N}$ by a time-weighted argument. Note that a time integral of the time weighted estimates of Alinhac good unknowns $U^\al$ for $\abs{\al}=4N$ will be obtained, which is crucial for closing the estimate for $\f$.
\begin{prop}\label{p_E_bound}
There exists $ \delta>0$ so that if $\g (T) \le \delta$, then
\begin{equation}\label{p_E_b_0}
 \frac{ \sem{2N}(t)}{(1+t)^{\vartheta}} +\int_0^t  \frac{ \sem{2N}(r)}{(1+r)^{1+\vartheta}} dr +\int_0^t\frac{\ns{ U^{4N}(r)}_{1} }{(1+r)^{ \vartheta}}dr \ls
\g(0) +   (\g(t))^{3/2}.
\end{equation}
\end{prop}
\begin{proof}
Multiplying \eqref{i_te_0} of Proposition \ref{i_h_evolution 2N} by $(1+t)^{-\vartheta}$ with $\vartheta>0$,  we find
\begin{align}\label{i_te_0'}
 &\dt \(\frac{\ns{  U^{4N}}_{0} + \sem{2N} }{(1+t)^{\vartheta}}\)+  \vartheta\frac{\ns{  U^{4N}}_{0} + \sem{2N} }{(1+t)^{1+\vartheta}}+\frac{\ns{ U^{4N}}_{1} }{(1+t)^{ \vartheta}} \nonumber
 \\& \quad \ls \frac{ \sqrt{\senn+\bar\k} \sem{2N}+ (\sqrt{\senn}+\k)\sdm{2N}+ \bar\k\f
 }{(1+t)^{ \vartheta}} 
.
\end{align}
By the estimate \eqref{bkes} of Lemma \ref{klem} and the definition \eqref{ggdef} of $\g(t)$, we obtain
\beq\label{decayk1}
\bar\k(t)\ls   \senn(t)  \ls \g(t) (1+t)^{-2 }.
\eeq
By \eqref{decayk1} and \eqref{decayk}, we deduce from \eqref{i_te_0'} that
\begin{align}\label{i_te_0'1}
 &\dt \(\frac{\ns{  U^{4N}}_{0} + \sem{2N} }{(1+t)^{\vartheta}}\)+  \vartheta\frac{\sem{2N} }{(1+t)^{1+\vartheta}}+\frac{\ns{ U^{4N}}_{1} }{(1+t)^{ \vartheta}} \nonumber
  \\& \quad \ls \sqrt{ \g}  \frac{ \sem{2N}}{(1+t)^{ 1+\vartheta}} +\sqrt{ \g}  \frac{  \sdm{2N}}{(1+t)^{ 1+\vartheta}} + \g\frac{  \f }{(1+t)^{ 2+\vartheta}} 
\end{align}
for $2-\kappa_d\ge 1$.
Integrating \eqref{i_te_0'1} directly in time, since $\delta>0$ is small, we have
\begin{align}\label{i_te_0''}
 & \frac{\sem{2N}(t)}{(1+t)^{\vartheta}} +   \int_0^t\frac{  \sem{2N}(r) }{(1+r)^{1+\vartheta}}dr+\int_0^t\frac{\ns{ U^{4N}(r)}_{1} }{(1+r)^{ \vartheta}}dr  \nonumber
  \\& \quad \ls \g(0) + \sqrt{ \g(t)}\int_0^t\(  \frac{  \sdm{2N}(r)}{(1+r)^{ 1+\vartheta}} +  \frac{  \f(r) }{(1+r)^{ 2+\vartheta}} \)dr\nonumber
 \\&\quad  \ls
\g(0) +   (\g(t))^{3/2} 
\end{align}
for $1+\vartheta\ge\vartheta+\kappa_d$. This yields \eqref{p_E_b_0}.
\end{proof}

Now we show a time weighted integral-in-time estimate of $ \sdm{2N} $.
\begin{prop} \label{Dgle2}
There exists $\delta>0$ so that if $\mathcal{G}_{2N} (T)\le\delta$, then
\begin{equation}\label{Dg2}
\int_0^t\frac{\sdm{2N}(r)}{(1+r)^{\vartheta+\kappa_d}}dr \lesssim
\fe{2N} (0) + (\g(t))^{3/2}.
\end{equation}
\end{prop}
\begin{proof}
 Multiplying \eqref{d2n} by $(1+t)^{-\vartheta-\kappa_d}$ of
Theorem \ref{dth} and then integrating in time,  by \eqref{p_E_b_0}, \eqref{Dg},  \eqref{decayk} and and the definition \eqref{ggdef} of $\g(t)$, we have
 \begin{align}
 \int_0^t\frac{\sdm{2N}(r)}{(1+r)^{\vartheta+\kappa_d}}dr & \lesssim  \int_0^t\frac{ \ns{ U^{4N} }_{1} + \fd{2N}+\k  \f}{(1+r)^{\vartheta+\kappa_d}}dr    \nonumber
 \\&\ls
\g(0) + (\g(t))^{3/2}+ \int_0^t  { \g(r)} \frac{  \f(r) }{(1+r)^{ 2+\vartheta}} dr\nonumber
 \\&\ls
\g(0) + (\g(t))^{3/2}
.
\end{align}
This gives \eqref{Dg2}.
\end{proof}

\subsection{Growth estimate  of  $\f$}
To estimate $\f$, it is crucial that we will use a different argument from \cite{GT_inf}.
Applying $\mathcal{J}^{4N+1/2}$  to the fourth equation in \eqref{geometric}, we have
\beq\label{feqe}
\dt \mathcal{J}^{4N+1/2}\eta+u_h\cdot D \mathcal{J}^{4N+1/2}\eta= \mathcal{J}^{4N+1/2}u_d -\[\mathcal{J}^{4N+1/2}, u_h\]\cdot D \eta.
\eeq

We now derive the estimate of  $\f$ by a time-weighted argument. 
\begin{prop}\label{p_f_bound}
There exists $\delta>0$ so that if $\mathcal{G}_{2N} (T)\le\delta$, then
\begin{equation}\label{p_f_b_0}
 \frac{ \f(t)}{(1+t)^{1+\vartheta}} +\int_0^t  \frac{ \f(r) }{(1+r)^{2+\vartheta }} dr \lesssim
\fe{2N} (0) + (\g(t))^{3/2}.
\end{equation}
\end{prop}
\begin{proof}
Taking the dot product of \eqref{feqe} with $\mathcal{J}^{4N+1/2}\eta$, and then integrate  by parts; using the commutator estimate of Lemma \ref{commutator} with $s=4N+1/2$, by \eqref{n213}, we find that
 \begin{align}\label{i_te_0'1}
  &\frac{1}{2}\frac{d}{dt}\f =\frac{1}{2}\frac{d}{dt}\int_\Sigma|\mathcal{J}^{4N+1/2}\eta|^2\nonumber \\&\quad=\int_\Sigma
  \mathcal{J}^{4N+1/2} u_d   \mathcal{J}^{4N+1/2}\eta
 -\frac{1}{2}\int_\Sigma u_h\cdott D |\mathcal{J}^{4N+1/2}\eta|^2-\int_\Sigma
     \left[\mathcal{J}^{4N+1/2},u_h\right]\cdott D \eta \mathcal{J}^{4N+1/2}\eta\nonumber
 \\&\quad=\int_\Sigma
  \mathcal{J}^{4N+1/2} u_d   \mathcal{J}^{4N+1/2}\eta
 +\frac{1}{2}\int_\Sigma D\cdott u_h |\mathcal{J}^{4N+1/2}\eta|^2-\int_\Sigma
     \left[\mathcal{J}^{4N+1/2},u_h\right]\cdott D \eta \mathcal{J}^{4N+1/2}\eta\nonumber
         \\&\quad\lesssim \abs{u_d }_{4N+1/2} \abs{  \eta}_{4N+1/2} +\abs{ D  u_h}_{L^\infty(\Sigma)}\abs{  \eta}_{4N+1/2}^2
\nonumber
 \\&\qquad+\left(\abs{ D  u_h}_{L^\infty(\Sigma)}\abs{  \eta}_{4N+1/2}+\abs{ 
u_h}_{4N+1/2}\abs{ D \eta}_{L^\infty(\Sigma)}\right)\abs{  \eta}_{4N+1/2} \nonumber
 \\&\quad\lesssim   \abs{u   }_{ 4N+1/2}  \sqrt{\f} +\abs{ D  u}_{L^\infty(\Sigma)}\f \nonumber
\\&\quad
\lesssim   \norm{U^{4N}}_{1} \sqrt{\f}+ \sqrt{\bar\k}\f
 \end{align}
Multiplying \eqref{i_te_0'1} by $(1+t)^{-1-\vartheta}$, by \eqref{decayk1} and the definition of $\g(t)$, we have
\begin{align}\label{i_te_0'2}
 \dt \(\frac{\f }{(1+t)^{1+\vartheta}}\)+ (1+ \vartheta)\frac{\f}{(1+t)^{2+\vartheta}}  \nonumber
   &   \ls \frac{   \norm{U^{4N}}_{1}  \sqrt{\f}+ \sqrt{\bar\k}\f  }{(1+t)^{ 1+\vartheta}}  \nonumber
  \\&   \ls \frac{   \norm{U^{4N}}_{1}  }{(1+t)^{ \vartheta/2}}\frac{ \sqrt{\f}  }{(1+t)^{ 1+\vartheta/2}} + \sqrt{\g} \frac{  \f  }{(1+t)^{ 2+\vartheta}}.
\end{align}
Integrating \eqref{i_te_0'2} directly in time, since $\delta>0$ is small and by Cauchy's inequality,  we deduce that, by \eqref{p_E_b_0},  
\begin{align}\label{i_te_0'3}
 \frac{\f (t)}{(1+t)^{1+\vartheta}} + \int_0^t\frac{\f(r)}{(1+r)^{2+\vartheta}}dr
  &  \ls \g(0)+\int_0^t \frac{   \norm{U^{4N}(r)}_{1}^2}{(1+r)^{ \vartheta }} dr \nonumber
  \\&\ls
\g(0) +   (\g(t))^{3/2}.
\end{align}
This yields \eqref{p_f_b_0}.
\end{proof}

\subsection{Decay estimate of
$\senn  $}
It remains to show the decay estimate of
$\senn  $.

\begin{prop} \label{decaylm}
There exists $\delta>0$ so that if
$\mathcal{G}_{2N} (T)\le\delta$, then
\begin{equation}\label{N+2,2}
(1+t)^{2} \senn  (t)\lesssim
\fe{2N} (0) +   (\g(T))^{2}\ \text{for all
}0\le t\le T.
\end{equation}
\end{prop}
\begin{proof}
Since $\fe{2N} (t)\le
{\mathcal{G}}_{2N} (T)\le \delta$, by taking $\delta$ small, we
obtain from \eqref{en+2} of Theorem \ref{eth} and \eqref{dn+2} of
Theorem \ref{dth} that
\begin{equation}\label{eesf}
\fe{N+2,2} \lesssim\feb{N+2,2}
\lesssim\fe{N+2,2}  \text{ and }
\fd{N+2,2} \lesssim\fdb{N+2,2}
\lesssim\fd{N+2,2} .
\end{equation}

Now we combine the estimate \eqref{p_u_e_11} of
Proposition \ref{p_upper_evolution N+2} and the estimate \eqref{tem en N+2} of Proposition \ref{i_temporal_evolution N+2}  to conclude that, by  the definition \eqref{i_horizontal_energy1} of $\seb{N+2,2} $ and the definition \eqref{i_horizontal_dissipation1} of $\fdb{N+2,2}$,
\begin{equation}  
   \frac{d}{dt} \left(\feb{N+2,2}   +2\int_\Omega J   \dt^{N+1} p  F^{2,N+2}\right)
+   \fdb{N+2,2}
  \ls \sqrt{\fe{2N}   } \sdnn 
\end{equation}
By  \eqref{tem en 2N00}, \eqref{eesf} and the smallness of $\delta$,  we deduce that there exists an instantaneous
energy, which is equivalent to $\seb{N+2,2} $ and $\senn  $, but for
simplicity is still denoted by $ \seb{N+2,2} $, such
that
\begin{equation} \label{u1}
\frac{d}{dt} \feb{N+2,2}  +  \fd{N+2,2} \le
0.
\end{equation}

In order to get decay from \eqref{u1}, we shall now estimate $\seb{N+2,2} $ in terms of $\fd{N+2,2} $. Notice that $\fd{N+2,2} $ can control every term in $ \seb{N+2,2} $ except $\as{D^{2(N+2)}\eta}_{0}$, $\abs{D^2\eta}_0^2$ and $\abs{\dt\eta}_0^2$. The key point is to use the Sobolev interpolation as in \cite{GT_inf}. Indeed, we first have that
\begin{equation} \label{intep0}
\as{D^{2(N+2)}\eta}_{0}\le  \abs{D^3\eta}_{2(N+2)-7/2} ^{2\theta} \abs{\eta}_{4N-1}^{2(1-\theta)}\lesssim   \fd{N+2,2}  ^\theta  \fe{2N}   ^{1-\theta},\text{
where }\theta=\frac{4N-10}{4N-9}.
\end{equation}
and
\begin{equation} 
\begin{split}
 \as{D^2 \eta}_{0}
&\le  \abs{\eta}_{0}^{2/3}  \abs{D^3\eta}_{0}^{4/3} \lesssim \fe{2N}   ^{1/3}  \fd{N+2,2} ^ {2/3}.
\end{split}
\end{equation}
By the $L^2$ table in \eqref{t2},
\begin{equation}\label{intep1}
 \as{\dt \eta}_{0} \ls   \fe{2N}   ^{1/3}  \fd{N+2,2} ^ {2/3}.
\end{equation}
 Hence, in light of \eqref{intep0}--\eqref{intep1}, since $N\ge 3$ and hence $(4N-10)/(4N-9)\ge 2/3$, we deduce
\begin{equation} \label{intep}
{\mathcal{E}}_{N+2,2} \lesssim\fe{2N}   ^{1/3}  \fd{N+2,2} ^ {2/3}.
\end{equation}

Denote by $\m:= \sup_{[0,T]}\fe{2N} $.
Hence by \eqref{u1} and \eqref{intep}, there exists some constant $C>0$
such that
\begin{equation}
\frac{d}{dt} \senn  +\frac{C}{\m^{1/2}}
 \senn  ^{3/2}\le 0.
\end{equation}
Solving this differential inequality directly, we obtain
\begin{equation} \label{u3}
\senn  (t)\le \frac{\m }{(\m^{1/2}
+ \hal C  (\senn  (0) )^{1/2} t)^{2} }
\fe{N+2,2} (0).
\end{equation}
Using that $\fe{N+2,2} (0)\lesssim\m  $, we obtain from \eqref{u3} that
\begin{equation}
\fe{N+2,2} (t)\lesssim
\frac{\m }{ 1+t^2 }  .
\end{equation}
This together with \eqref{Dg} directly implies \eqref{N+2,2}.
\end{proof}

\section{Proof of main theorem}
Now we can arrive at our ultimate energy estimates for
$\mathcal{G}_{2N} $.
\begin{thm}\label{Ap}
There exists a universal $0 < \delta < 1$ so that if $
\mathcal{G}_{2N} (T) \le \delta$, then
\begin{equation}\label{Apriori}
 \mathcal{G}_{2N} (T) \ls\mathcal{G}_{2N} (0).
 \end{equation}
\end{thm}
\begin{proof}
It follows directly from the definition of
$\mathcal{G}_{2N} $ and Propositions
\ref{Dgle}--\ref{decaylm} that
\begin{equation}
 \mathcal{G}_{2N} (t) \ls\mathcal{G}_{2N} (0) + (\mathcal{G}_{2N}  (T))^{3/2}\text{ for all }0 \le t \le
 T.
\end{equation}
This proves \eqref{Apriori} since $\delta>0$ is small.
\end{proof}

Now we can present the
 \begin{proof}[Proof of Theorem \ref{intro_inf_gwp}]
By the a priori estimates \eqref{Apriori} of Theorem \ref{Ap}, we may employ a continuity argument as in Section 11 of \cite{GT_inf} to prove that there exists a universal $\varepsilon_0>0$ such that if $\fe{2N} (0)+\f(0)\le \varepsilon_0$, then the unique local solution constructed in \cite{GT_lwp} is indeed a  global solution to \eqref{geometric} on $[0,\infty)$. This completes the proof of Theorem \ref{intro_inf_gwp}.
\end{proof}

\appendix
\section{Analytic tools}

\subsection{Poisson integral }
For a function $f$, defined on $\Sigma = \Rn{d-1}$, the Poisson integral in $\Rn{d-1} \times (-\infty,0)$  is defined by
\begin{equation}\label{poisson_def_inf}
 \mathcal{P}f(x',x_d ) = \int_{\Rn{d-1}} \hat{f}(\xi) e^{2\pi \abs{\xi}x_d } e^{2\pi i x' \cdot \xi} d\xi.
\end{equation}
Although $\mathcal{P} f$ is defined in all of $\Rn{d-1} \times (-\infty,0)$, we will only need bounds on its norm in the restricted domain $\Omega = \Rn{d-1} \times (-b,0)$.  This yields a couple improvements of the usual estimates of $\mathcal{P} f$ on the set $\Rn{d-1} \times (-\infty,0)$.

\begin{lem}\label{p_poisson}
Let $\mathcal{P} f$ be the Poisson integral of a function $f$ that is either in $\dot{H}^{q}(\Sigma)$ or $\dot{H}^{q-1/2}(\Sigma)$ for $q \in \mathbb{N}$, where $\dot{H}^s$ is the usual homogeneous Sobolev space of order $s$.  Then
\begin{equation}
 \ns{\nab^q \mathcal{P}f }_{0} \ls \norm{f}_{\dot{H}^{q-1/2}(\Sigma)}^2 \text{ and }  \ns{\nab^q \mathcal{P}f }_{0} \ls \norm{f}_{\dot{H}^{q}(\Sigma)}^2.
\end{equation}
\end{lem}
\begin{proof}
It follows from the definition \eqref{poisson_def_inf}. We refer to Lemma A.7 of \cite{GT_lwp}.
\end{proof}
\subsection{Interpolation estimates}

Assume that $\Sigma = \Rn{d-1}$ and $\Omega = \Sigma \times (-b,0)$. We begin with an interpolation result for Poisson integrals, as defined by \eqref{poisson_def_inf}.

\begin{lem}\label{i_poisson_interp}
Let $\mathcal{P} f$ be the Poisson integral of $f$, defined on $\Sigma$. Let   $q,s \in \mathbb{N}$, and $r\ge 0$.   Then the following estimates hold.

\begin{enumerate}
 \item Let
\begin{equation}\label{i_p_i_1}
 \theta = \frac{s}{q+ s } \text{ and } 1-\theta = \frac{q }{q+s }.
\end{equation}
Then
\begin{equation}\label{i_p_i_2}
 \ns{\nab^q \mathcal{P}f }_{0} \ls \left( \as{  f }_{0}  \right)^\theta  \left( \as{D^{q+s} f}_{0}  \right)^{1-\theta}.
\end{equation}

\item
Let $r+s >(d-1)/2$,
\begin{equation}\label{i_p_i_3}
\theta = \frac{r+s-1}{q+s+r }, \text{ and } 1-\theta = \frac{q +1}{q+s+r }.
\end{equation}
Then
\begin{equation}\label{i_p_i_4}
 \pns{\nab^q \mathcal{P}f }{\infty} \ls \left( \as{  f}_{0}  \right)^\theta  \left( \as{D^{q+s} f }_{r}  \right)^{1-\theta}.
\end{equation}

\item Let $s>(d-1)/2$.  Then
\begin{equation}\label{i_p_i_5}
 \pns{\nab^q \mathcal{P}f }{\infty} \ls \as{D^q f}_{s}.
\end{equation}

\end{enumerate}
\end{lem}
\begin{proof}
We refer to Lemma A.6 of \cite{GT_inf}.
\end{proof}

The next result is a similar interpolation result for functions defined only on $\Sigma$.

\begin{lem}\label{i_sigma_interp}
Let $f$ be  defined on $\Sigma$.  Then the following estimates hold.

\begin{enumerate}
 \item Let $q,s \in (0,\infty)$ and
\begin{equation}
 \theta = \frac{s}{q+ s } \text{ and } 1-\theta = \frac{q }{q+s }.
\end{equation}
Then
\begin{equation}\label{i_sig_i_1}
 \as{D^q  f }_{0} \ls \left( \as{  f }_{0}  \right)^\theta  \left( \as{D^{q+s} f}_{0}  \right)^{1-\theta}.
\end{equation}

\item
Let $q,s \in \mathbb{N}$,  $r\ge 0$,  $r+s >(d-1)/2$,
\begin{equation}
\theta = \frac{r+s-1}{q+s+r }, \text{ and } 1-\theta = \frac{q +1}{q+s+r }.
\end{equation}
Then
\begin{equation}\label{i_sig_i_2}
 \as{D^q  f }_{L^\infty} \ls \left( \as{  f}_{0}  \right)^\theta  \left( \as{D^{q+s} f }_{r}  \right)^{1-\theta}.
\end{equation}

\end{enumerate}
\end{lem}
\begin{proof}
See Lemma A.7 of \cite{GT_inf}.
\end{proof}

Now we record a similar result for functions defined on $\Omega$ that are not Poisson integrals.  The result follows from estimates on fixed horizontal slices.

\begin{lem}\label{i_slice_interp}
Let $f$ be a function on $\Omega$.   Let $q,s \in \mathbb{N}$, and $r\ge 0$.   Then the following estimates hold.
\begin{enumerate}
 \item Let
\begin{equation}\label{i_sl_i_1}
 \theta = \frac{s}{q+ s } \text{ and } 1-\theta = \frac{q }{q+s }.
\end{equation}
Then
\begin{equation}\label{i_sl_i_2}
 \ns{D^q  f }_{0} \ls \left( \ns{  f }_{0}  \right)^\theta  \left( \ns{D^{q+s} f}_{0}  \right)^{1-\theta}.
\end{equation}

\item
Let $r+s >(d-1)/2$,
\begin{equation}\label{i_sl_i_3}
\theta = \frac{r+s-1}{q+s+r }, \text{ and } 1-\theta = \frac{q +1}{q+s+r }.
\end{equation}
Then
\begin{equation}\label{i_sl_i_4}
 \pns{D^q  f }{\infty} \ls \left( \ns{  f}_{1}  \right)^\theta  \left( \ns{D^{q+s} f }_{r+1}  \right)^{1-\theta}
\end{equation}
and
\begin{equation}\label{i_sl_i_05}
 \ns{D^q  f }_{L^\infty(\Sigma)} \ls \left( \ns{  f}_{1}  \right)^\theta  \left( \ns{D^{q+s} f }_{r+1}  \right)^{1-\theta}
\end{equation}

\end{enumerate}
\end{lem}

\begin{proof}
We refer to Lemma A.8 of \cite{GT_inf}.
\end{proof}
\subsection{Poincar\'{e}-type  inequalities }
We have the following Poincar\'{e}-type  inequalities.

\begin{lem}\label{poincare_b}
 It holds that
\beq
 \norm{f}_{0}  \ls  \abs{f}_{0}+ \norm{\p_d f}_{0}
\text{
and
}
 \norm{f}_{L^\infty}  \ls  \abs{f}_{L^\infty} + \norm{\p_d f}_{L^\infty} .
\eeq
\end{lem}

\begin{lem}\label{poincare_trace}
For $f=0$ on $\Sigma_b$, it holds that 
\beq
 \norm{f}_{0}  \ls   \norm{\p_d f}_{0}
\text{
and
}
 \norm{f}_{L^\infty}  \ls   \norm{\p_d f}_{L^\infty} .
\eeq
\end{lem}

\begin{lem}\label{poincare_usual}
For $f=0$ on $\Sigma_b$, it holds that 
\beq
\norm{f}_{0} \ls \norm{f}_{1} \ls \norm{\nab f}_{0}
\text{
and
}
\norm{f}_{L^\infty} \ls\norm{f}_{W^{1,\infty}}\ls \norm{\nabla f}_{L^\infty}.
\eeq
\end{lem}

We will need a version of Korn's inequality.
\begin{lem}\label{i_korn}
It holds that $\norm{u}_{1} \ls \norm{\sg u}_{0}$ for all $u \in H^1(\Omega;\Rn{d})$ so that $u=0$ on $\Sigma_b$.
\end{lem}
\begin{proof}
We refer to Lemma 2.7 of \cite{beale_1}.
\end{proof}

\subsection{Product estimates}
We will need some estimates of the product of functions in Sobolev spaces.

\begin{lem}\label{i_sobolev_product_1}
Let the domain of function spaces be either $\Sigma$ or $\Omega$.
 Let $0\le r \le s_1 \le s_2$ be such that  $s_2 >r+ n/2$, $n=d-1$ or $d$, then 
 \begin{equation}\label{i_s_p_02}
 \norm{fg}_{ H^r} \lesssim \norm{f}_{ H^{s_1}} \norm{g}_{ H^{s_2}}.
\end{equation}
\end{lem}
\begin{proof}
We refer to Lemma A.1 of \cite{GT_inf}.
\end{proof}

We will also need the following variant.
\begin{lem}\label{i_sobolev_product_2}
It holds that
\begin{equation}
 \abs{ fg}_{1/2} \ls \abs{f}_{C^1(\Sigma)} \abs{g}_{1/2}
\text{
and 
 }
 \abs{ fg}_{-1/2} \ls \abs{f}_{C^1(\Sigma)} \abs{g}_{-1/2}.
\end{equation}
\end{lem}
\begin{proof}
We refer to Lemma A.2 of \cite{GT_inf}.
\end{proof}

\subsection{Commutator estimate}

Let $\mathcal{J}= {(1-\Delta)}^{1/2}$ with $\Delta$  the Laplace operator on $\Rn{n}$.  

We recall the following commutator estimate:
 \begin{lem}\label{commutator}
\begin{equation}
\|[\mathcal{J}^{s},f]g\|_{L^2}\le \|\nabla f\|_{L^\infty}\|\mathcal{J}^{s-1}g\|_{L^2}
+ \|\mathcal{J}^s f\|_{L^2}\| g\|_{L^\infty}.
\end{equation}
\end{lem}
\begin{proof}
We refer to Lemma X1 of \cite{KP}.
\end{proof}

\subsection{Stokes elliptic estimates}

We recall the elliptic estimates for two classical Stokes problems with different boundary conditions.
\begin{lem}\label{i_linear_elliptic}
Let $r \ge 2$.  Suppose that $f\in H^{r-2}(\Omega), h\in H^{r-1}(\Omega)$ and $\psi\in H^{r-3/2}(\Sigma)$. Then there exists unique $u\in H^r(\Omega), p\in H^{r-1}(\Omega)$  solving the problem
\begin{equation}\label{stokes1}
 \begin{cases}
  -\Delta u + \nab p =f  &\text{in }\Omega \\
  \diverge{u} = h &\text{in }\Omega \\
  (pI_d- \sg u  )e_d  = \psi &\text{on }\Sigma \\
  u =0 &\text{on }\Sigma_b.
 \end{cases}
\end{equation}
 Moreover,
\begin{equation}
 \ns{u}_{r} + \ns{p}_{r-1}  \ls \ns{f}_{r-2}  + \ns{h}_{r-1} + \abs{\psi}_{r-3/2}^2 .
\end{equation}
\end{lem}
\begin{proof}
We refer to the proof in Lemma 3.3 of \cite{beale_1}.
\end{proof}

\begin{lem}\label{i_linear_elliptic2}
Let $r \ge 2$.  Suppose that $f\in H^{r-2}(\Omega), h\in H^{r-1}(\Omega), \varphi\in H^{r-1/2}(\Sigma)$ and that  $u\in H^r(\Omega), p\in H^{r-1}(\Omega)$ (up to constants) solving the problem
\begin{equation}\label{stokes2}
 \begin{cases}
  -\Delta u + \nab p =f  &\text{in }\Omega \\
  \diverge{u} = h &\text{in }\Omega \\
u= \varphi &\text{on }\Sigma \\
  u =0 &\text{on }\Sigma_b.
 \end{cases}
\end{equation}
Then
\begin{equation}
 \norm{u}_{r}^2 + \norm{\nab p}_{r-2}^2 \ls \norm{u}_{0}^2 + \norm{f}_{r-2}^2 + \norm{h}_{r-1}^2 + \abs{\varphi}_{r-1/2}^2.
\end{equation}
\end{lem}
\begin{proof}
We refer to the proof of (3.7) in \cite{K}.
\end{proof}

\end{document}